\documentclass[11pt]{article} 

\usepackage[utf8]{inputenc} 
  
\usepackage{geometry} 
\usepackage{graphicx} 

\usepackage{booktabs} 
\usepackage{array} 
\usepackage{paralist} 
\usepackage{verbatim} 
\usepackage{amsmath}
\usepackage{amssymb}
\usepackage{amsthm}
\usepackage{xifthen}
\usepackage{xcolor}
\usepackage{ulem}
\usepackage{comment}

\usepackage{subfigure}

\usepackage{soul}

\usepackage{fancyhdr} 
\usepackage{sectsty}
\allsectionsfont{\sffamily\mdseries\upshape} 

\newtheorem{theorem}{Theorem}
\newtheorem*{acknowledgement*}{Acknowledgement}

\newtheorem{definition}[theorem]{Definition}

\newtheorem{lemma}[theorem]{Lemma}

\newtheorem{proposition}[theorem]{Proposition}
\newtheorem{remark}[theorem]{Remark}

\usepackage[nottoc,notlof,notlot]{tocbibind} 
\usepackage[titles,subfigure]{tocloft} 

\def\R{\mathbb{R}}

\def\N{\mathbb{N}}

\DeclareMathOperator*{\argmax}{argmax}
\DeclareMathOperator{\E}{\mathbb E}

\newcommand{\kom}[1]{}
\renewcommand{\kom}[1]{{\bf [#1]}}

\usepackage{accents}
\newcommand{\ubar}[1]{\underaccent{\bar}{#1}}
\newcounter{komcounter}
\numberwithin{komcounter}{section}

\title{A controller-stopper-game with hidden controller type}
\author{
 Andi Bodnariu\\Department of Mathematics, Stockholm University\\ 
\\
Kristoffer Lindensjö\\Department of Mathematics, Stockholm University
}

\begin{document}

\maketitle

\textbf{Abstract:} 
We consider a continuous time stochastic dynamic game between a stopper (Player $1$, the \textit{owner} of an asset yielding an income) and a controller (Player $2$, the \textit{manager} of the asset), where the manager is either effective or non-effective. An effective manager can choose to exert low or high effort which corresponds to a high or a low positive drift for the accumulated income of the owner with random noise in terms of Brownian motion; where high effort comes at a cost for the manager. The manager earns a salary until the game is stopped by the owner, after which also no income is earned. A non-effective manager cannot act but still receives a salary. For this game we study (Nash) equilibria using stochastic filtering methods; in particular, in equilibrium the manager controls the learning rate (regarding the manager type) of the owner. First, we consider a strong formulation of the game  which requires restrictive assumptions for the admissible controls, and find an equilibrium of (double) threshold type. Second, we consider a weak formulation, where a general set of admissible controls is considered. We show that the threshold equilibrium of the strong formulation is also an equilibrium in the weak formulation.

\section{Introduction}\label{sec:intro}
We consider a continuous time two player stochastic game between a 
stopper (Player 1) and a controller (Player 2). 
The controlled process $(X_t)$ is given by 
\begin{align} \label{control-process}
X_t=\int_0^t\left(\theta\lambda_s-c\right)ds+W_t,
\end{align}
where 
$(\lambda_t)$ is a process chosen by the controller, 
$(W_t)$ is a Brownian motion, 
$\theta$ is an independent Bernoulli random variable with $\mathbb{P}(\theta=1) = 1-\mathbb{P}(\theta=0) =p \in (0,1)$ indicating whether the controller is  effective (or active) or not,
and $c > 0$ is a constant. 
On the other, based on the observations of $(X_t)$, the stopper selects a stopping time $\tau$ at which the game ends. 

For a given stopping-control strategy pair $(\tau,(\lambda_t))$ the reward of the 
stopper is 
\begin{align} \label{reward-employer}
\mathcal{J}^1\left(\tau,(\lambda_t),p\right)=\E\left[\int_0^\tau e^{-rs}\left(\theta \lambda_s-c\right)ds\right]
\end{align}
and the reward of the controller is 
\begin{align} \label{reward-employee}
\mathcal{J}^2\left(\tau,(\lambda_t),p\right)=\E\left[\int_0^\tau e^{-rs}\left(c-(\lambda_s-\ubar{\lambda})^2\right)ds\Bigg|\theta=1\right],
\end{align}
where $r>$ is a constant (discount rate). 

The control process values are restricted to take one of two constants $\{\bar{\lambda},\ubar{\lambda}\}$ at each time $t$, where we assume that $\bar{\lambda}>\ubar{\lambda} > c > 0$. 
The model is further specified in Sections \ref{Strong_sol_approach} and \ref{sec:weak} where we also define notions of (Nash) equilibria corresponding to both players wishing to maximize their respective rewards.

The interpretation is that $(X_t)$ is the accumulated income of Player $1$ who wants to maximize the discounted accumulated income by selecting a time $\tau$ at which the game ends. This income has a positive drift if $\theta(\omega)=1$ and a negative drift if $\theta(\omega)=0$; however, the outcome of $\theta$ cannot be observed by Player $1$ who must make the stopping decision based only on observations of $(X_t)$. The stopping decision will in equilibrium, as we will show, be made based on the probability that Player $1$ assigns to the event $\{\theta=1\}$, which is dynamically updated based on the observations of $(X_t)$. 

On the other hand, an active Player $2$, i.e., in case $\theta(\omega)=1$, can affect the drift of the income $(X_t)$ by dynamically selecting the effort level, i.e., 
$(\lambda_t)$, and thereby, as we shall see, affect the probability that Player $1$ assigns to the event $\{\theta=1\}$. 
The accumulated income of an active Player $2$ is based on the constant income rate $c$ minus the cost rate $(\lambda_t-\ubar{\lambda})^2$ which is zero when effort is low and positive otherwise; cf. \eqref{reward-employee}. Moreover, an active Player $2$ wants to dynamically select the effort level $(\lambda_t)$ in order to maximize the discounted accumulated income of Player $2$ until Player $1$ ends the game. Player $2$ must therefore consider the trade-off between exerting a large effort, which is costly, and a small effort, which implies no cost but decreases the probability that Player $1$ assigns to $\{\theta=1\}$ compared to the large effort. An inactive Player $2$ cannot act at all.

In line with the interpretation above, our ansatz to this problem is to use stochastic filtering methods to search for an equilibrium which depends on the conditional probability of the event $\{\theta=1\}$ based on the observations of $(X_t)$, which corresponds to Player 1's continuously updated belief about Player 2 being active. 
Indeed, we find such an equilibrium of (double) threshold type meaning that we find two thresholds $0<b_1^*<b_2^*<1$ such that an equilibrium is that 
Player $2$ exerts the smaller effort $\ubar{\lambda}$ when the conditional probability of $\{\theta=1\}$ is above $b_2^*$ and the larger 
effort $\bar{\lambda}$ when the conditional probability is below $b_2^*$, 
and Player $1$ stops the game whenever the conditional probability of $\{\theta=1\}$ falls below $b_1^*$; see Remark \ref{rem:threshold.inter} for details.

We study this game in an increasing order of generality regarding the set of admissible control strategies. First using a \textit{strong formulation} and second using a \textit{weak formulation} of the game. The same threshold equilibrium is obtained in both formulations.

\begin{itemize}
\item \textbf{Strong formulation}: 
In the strong formulation we admit control strategies only of Markovian type in the sense that $\lambda_t = \lambda(P_t)$, where $\lambda:[0,1]\rightarrow\{\bar{\lambda},\ubar{\lambda}\}$ (a deterministic function) and $(P_t)$ is defined as a process which in equilibrium coincides with the conditional probability that the stopper assigns to $\{\theta=1\}$. The process $(P_t)$ is here the strong solution to a particular stochastic differential equation; see the beginning of Section \ref{Strong_sol_approach} for details. 
In this formulation the main results are: 
(i) we provide a verification theorem for a double threshold equilibrium, 
and (ii) we prove that a double threshold equilibrium exists under certain parameter restrictions.

\item \textbf{Weak formulation}: In the weak formulation, admissible control strategies correspond to a general set of stochastic processes adapted to a filtration generated by $(X_t)$ taking values in $\{\ubar{\lambda},\bar{\lambda}\}$. Here, however, we start by defining $(X_t)$ as a Brownian motion and we achieve  a controlled process analogous to the one in \eqref{control-process} by means of a measure change, with which we define the reward functions and a corresponding equilibrium; see Section \ref{sec:weak} for details. The main result is that the double threshold equilibrium found in the strong formulation is also an equilibrium in the weak formulation, i.e., when allowing a larger set of admissible control strategies.

\end{itemize}
In Section \ref {prevlit} we survey related previous literature and clarify the contribution of the present paper. 
In Section \ref{filter_theory} we present stochastic filtering arguments which are relevant to the subsequent sections.
The strong formulation of our game is studied in Section \ref{Strong_sol_approach}. 
In particular, the beginning of 
Section \ref{Strong_sol_approach} specifies the strong formulation further, 
Section \ref{sec:search-NE} contains a heuristic derivation of an equilibrium candidate, 
Section \ref{sec:ver-thm} reports the verification result, and 
Section \ref{sec:existNE} reports the equilibrium existence result. The weak formulation is studied in Section \ref{sec:weak}.

\subsection{Previous literature and contribution}\label{prevlit}
The problem studied in the present paper belongs to a new class of dynamic stochastic control and stopping games with the key feature being that the player's may be \text{ghosts} 
(cf. \cite{de2020playing}) in the sense that a player does not necessarily exist, or equivalently is not activity, or not effective.  
This ghost feature was first studied in \cite{de2020playing}, where a two-player stopping game is studied and the term ghost was introduced. 
In \cite{ekstrom2022detect}, a controller-stopper-game where the stopper faces unknown competition in the form of a ghost controller is studied, in the context of a fraud detection application. 
In \cite{ekstrom2022finetti}, a de Finetti controller-stopper-game (of resource extraction) where the controller faces unknown competition in the form of a stopper ghost with the option to extract all the remaining resources instantaneously is studied.

From a game theoretic interpretation our main contribution is that our game is a non-zero-sum game where the player objectives agree in the sense that both players would benefit if the hidden controller were revealed (in the case of an active Player $2$). In this sense, both players are not exactly competing against each other, but rather aiming on finding an agreement that would benefit both. Typically, such situations are complicated since it makes existence of (non-trivial) Nash equilibria sensitive to the specific player payoffs.  
This stands in contrast to previously studied games of these type, see e.g., \cite{ekstrom2022detect} where the profit of one player is an immediate loss for the other, which results in opposite player objectives in the sense that the controller aims at staying hidden, which is the opposite to our situation.

From a technical view-point, our main contribution is twofold. 
First, we constrain the control process to take values in a finite set, i.e., $\{\bar{\lambda},\ubar{\lambda}\}$. 
This means that we can interpret the problem  of the controller as an optimal switching problem without a cost for switching, implying that switching (between the two control values $\{\bar{\lambda},\ubar{\lambda}\}$) may occur infinitely often, which stands in contrast to the usual formulation of optimal switching problems; see e.g., \cite{olofsson2021management} and the references therein. 
Second, we consider a weak formulation for these types of games, based on defining the state process $(X_t)$ as a Brownian motion, and the reward functions in terms of measure changes. 
Then we show that the Nash equilibrium in Markovian strategies (i.e., in the strong formulation), is also a Nash equilibrium in the weak formulation. 

This weak approach is inspired by \cite{erik2022}, which formulates a weak approach in the study of a sequential estimation problem, where the optimizer can choose a bounded control representing the rate at which the information is received and a stopping time at which the experiment ends, in particular, \cite{erik2022} considers an optimization problem and not a game. Weak solution approaches to dynamic stochastic games have previously been considered in a variety of recent papers; 
see e.g., \cite{doi:10.1137/120894907}, and \cite{possamai2020zero} which contain surveys of the related literature.  


In a broader context, the problem studied in the present paper can be regarded as a controller-stopper-game under incomplete information. 
Controller-stopper-games were first studied for zero-sum games. In \cite{karatzas2001controller} a zero-sum game between a controller and a stopper is studied for a one-dimensional diffusion, 
whereas \cite{bayraktar2013multidimensional} considers the game in a multidimensional setting. In \cite{nutz2015optimal} a zero-sum game between a stopper and a controller choosing a probability measure is studied. Singular controls for zero-sum controller-stopper-games were studied in \cite{hernandez2015zero} for a one-dimensional diffusion and in \cite{bovo2022variational,bovo2023zero} for the multidimensional setting. In \cite{hernandez2015games} zero-sum controller-stopper-games with singular control are studied for a spectrally one-sided Lévy process. A zero-sum game between a stopper and a player controlling the jumps of the state process is studied in \cite{bayraktar2019controller}. 

Stochastic games under asymmetric information were first considered in \cite{cardaliaguet2009stochastic}, which considers a zero-sum stochastic differential game between two controllers. In \cite{cardaliaguet2013pathwise} path-wise non-adaptive controls are studied for a zero-sum game between two controllers. 
An asymmetric information Dynkin game with a random expiry time observed by one of the players is studied in \cite{lempa2013dynkin}. 
In \cite{gensbittel2018two} a two-player zero-sum game under asymmetric information is considered where only one player can observe the underlying Brownian motion, while the second the player only observes the strategy chosen by the first player. A zero-sum game where both players observe different processes is studied in \cite{gensbittel2019zero}. 
Non-Markovian zero-sum games under partial information are also considered; see e.g., \cite{de2022value}.

For a background regarding the interpretation of our game as a dynamic signaling game between an owner (Player 1, the stopper) and a manager (Player 2, the controller) see \cite{ekstrom2023hiring} and the references therein.

\subsection{The underlying stochastic filtering theory arguments}\label{filter_theory}
The present section contains a brief account of the stochastic filtering arguments that underlie the analysis of the present paper. 
The section is included as an informal and heuristic precursor for content of the subsequent sections. 
A formal result in the direction of this section is Proposition \ref{filter_prop}.  

Let us first consider the perspective of the stopper. Assuming that the controller uses a control strategy $(\lambda^*_t)$ we obtain---using standard filtering theory; see e.g., \cite[Chapter 8.1]{liptser2001statistics}---that the innovations process defined by 
\begin{align*} 
\hat{W}_t=X_t+ct-\int_0^t\E[\lambda^*_s\theta|\mathcal{F}^X_s]ds
\end{align*}
is a Brownian motion with respect to $((\mathcal{F}^X_t),\mathbb{P})$, 
where $(\mathcal{F}^X_t)$ is defined as the smallest right-continuous filtration to which $(X_t)$ is adapted.

Relying again on basic filtering theory, and arguments similar to those in \cite[Section 2.1]{ekstrom2022detect}, we find that if the strategy $(\lambda^*_t)$ is  $(\mathcal{F}^{X}_t)$-adapted---we shall later see that an equilibrium with this property can indeed be found---then the conditional probability (process) that the stopper assigns to the controller being active, i.e., $\mathbb{P}(\theta=1|\mathcal{F}^X_t)=\mathbb{E}[\theta|\mathcal{F}^X_t], t \geq 0$ is given by the 
stochastic differential equation (SDE)
\begin{align}\label{P-eq-equlibrium}
dP_t=\lambda^*_tP_t(1-P_t)d\hat{W}_t, \enskip P_0=p.
\end{align}
Note that the observations above rely implicitly on the assumption that the control strategy $(\lambda^*_t)$ is fixed in the sense 
that the stopper knows which process ($\lambda^*_t$) that the controller uses. However, in order to verify that a candidate equilibrium strategy $(\lambda^*_t)$  is indeed an equilibrium strategy (cf. Definition \ref{def:NE} below) we must be able to analyze what happens to an equilibrium stopping strategy---which as we shall see will be determined as a threshold time in terms of the conditional probability process---when the controller deviates from the candidate equilibrium strategy. 

To this end observe that if we consider an $(\mathcal{F}^{X}_t)$-adapted candidate equilibrium strategy $(\lambda^*_t)$ and an arbitrary admissible deviation (control) strategy $(\lambda_t)$, and now \textit{define} a process $(P_t)$ to be given by
\begin{align}
dP_t =\lambda^*_tP_t(1-P_t)(\theta\lambda_t-\lambda^*_tP_t)dt+\lambda^*_tP_t(1-P_t)dW_t, \enskip P_0 = p, \label{P_eq-deviation}
\end{align}
then $(P_t)$ depends, of course, on the equilibrium candidate $(\lambda^*_t)$ as well as the deviation strategy $(\lambda_t)$. However, using the observations above it is also directly verified that $P_t = \mathbb{P}(\theta=1|\mathcal{F}^X_t), t \geq 0$ in the special case of no deviation (i.e., with $(\lambda^*_t)=(\lambda_t)$). In other words, $(P_t)$ defined as in \eqref{P_eq-deviation} coincides with the conditional probability process in the case of no deviation, but it also tells us how the controller affects $(P_t)$ in the case of deviation, and we may therefore, as we will see, use this definition of $(P_t)$ to find an equilibrium.

\section{Strong formulation}\label{Strong_sol_approach}
Let $(\Omega,\mathcal{F},\mathbb{P})$ be a probability space supporting the standard one-dimensional Brownian motion $(W_t)$ and the independent Bernoulli random variable $\theta$, where we recall that $\mathbb{P}(\theta=1) = 1- \mathbb{P}(\theta=1)=p \in (0,1)$.

Observe that if both the candidate equilibrium strategy and the deviation strategy in \eqref{P_eq-deviation} are of Markov control type, 
specifically in the sense  that
$\lambda^*_t=\lambda^*(P_t)$ and $\lambda_t=\lambda(P_t)$ where 
$\lambda^*,\lambda:[0,1]\rightarrow\{\bar{\lambda},\ubar{\lambda}\}$, then $P_t = P^{\lambda,\lambda^*}_t$ in \eqref{P_eq-deviation} will be given by the SDE
\begin{align}
dP_t=\lambda^*(P_t)P_t(1-P_t)(\theta\lambda(P_t)-\lambda^*(P_t)P_t)dt+\lambda^*(P_t)P_t(1-P_t)dW_t, \enskip P_0=p. \label{P_def_strong_sol}
\end{align}
(Depending on the context we will, to ease notation, sometimes write $P_t$ and sometimes write $P_t^{\lambda,\lambda^*}$.)

In the present section we will restrict the set of admissible control strategies to be of Markov control type. 
Recall  that Section \ref{sec:weak} contains a weak formulation of our game where we relax the notion of admissible strategies to be a set of general stochastic processes (taking values in $\{\bar{\lambda},\ubar{\lambda}\}$). By restricting to  Markov controls we ensure that $(P_t)$ is obtained as the strong solution to  \eqref{P_def_strong_sol}; see Proposition \ref{prop_strong_ex} in Appendix \ref{app:SDE}.
Furthermore, using the definition of $(X_t)$ in \eqref{control-process} as well as \eqref{P_def_strong_sol} we note that the dynamics of $(P_t)$ can be written as
\begin{align}\label{XP_SDE}
dP_t&=\lambda^*(P_t)P_t(1-P_t)(c-\lambda^*(P_t)P_t)dt+\lambda^*(P_t)P_t(1-P_t)dX_t, \enskip P_0=p,
\end{align} 
and that $(P_t)$ is $(\mathcal{F}^X_t)$-adapted. 
Formally, we restrict the set of admissible control strategies to be of Markov control type by identifying an 
admissible control strategy $(\lambda_t)$ with a deterministic function 
$\lambda:[0,1]\rightarrow\{\bar{\lambda},\ubar{\lambda}\}$ according to $\lambda_t=\lambda(P_t)$, where $(P_t)$ is given by \eqref{P_def_strong_sol}, and where $\lambda$ satisfies the conditions of Definition \ref{def:admissibility} (which also defines the set of admissible stopping strategies). 

\begin{definition}[Admissibility in the strong formulation] \label{def:admissibility}
\enskip
\begin{itemize}
\item A Markov control (deterministic function) $\lambda:[0,1]\rightarrow \{\ubar{\lambda},\bar{\lambda}\}$ is said to be an admissible control strategy if it is RCLL (right-continuous with left hand limits). 
The set of admissible control strategies is denoted by $\mathbb{L}$. 

\item A stopping time $\tau$ is said to be an admissible stopping strategy if it is adapted to $(\mathcal{F}^X_t)$. 
The set of admissible stopping strategies is denoted by $\mathbb{T}$. 
 \end{itemize}
\end{definition}
To clarify, a control process $(\lambda_t)$ is obtained by
\begin{align*}
\lambda_t =\lambda(P^{\lambda,\lambda^*}_t), \enskip t \geq 0, 
\end{align*}
where $P^{\lambda,\lambda^*}_t = P_t$, with $(\lambda,\lambda^*)\in \mathbb{L}^2$, is given by \eqref{P_def_strong_sol}; i.e., 
a control process $(\lambda_t)$ depends generally on a \textit{pair} of admissible strategies 
$(\lambda^*,\lambda)\in \mathbb{L}^2$ which represents the candidate equilibrium strategy $\lambda^*$ and the deviation strategy $\lambda$, respectively.

In line with Section \ref{sec:intro}, both players want to maximize their respective rewards and we define our Nash equilibrium accordingly.
\begin{definition}[Nash equilibrium] \label{def:NE}
 A pair of admissible strategies $(\tau^*,\lambda^*) \in \mathbb{T} \times \mathbb{L}$ is a said to be a Nash equilibrium if 
the corresponding rewards, \eqref{reward-employer}--\eqref{reward-employee}, satisfy
\begin{align}
\begin{cases}\label{NE_cond}
\mathcal{J}^1\left(\tau^*,(\lambda^*(P^{\lambda^*,\lambda^*}_t)),p\right)\geq \mathcal{J}^1\left(\tau,(\lambda^*(P_t^{\lambda^*,\lambda^*})),p\right),\\
\mathcal{J}^2\left(\tau^*,(\lambda^*(P_t^{\lambda^*,\lambda^*})),p\right)\geq \mathcal{J}^2\left(\tau^*,(\lambda(P_t^{\lambda,\lambda^*})),p\right),
\end{cases}
\end{align}
for any pair of deviation strategies $(\tau,\lambda) \in \mathbb{T} \times \mathbb{L} $. 
\end{definition}

\begin{remark}  
In line with the usual interpretation of a Nash equilibrium we note that the 
first condition in \eqref{NE_cond} implies that deviating from the equilibrium is sub-optimal for the stopper, 
and that the second condition implies the same for the controller. 
Note also that the appearance of the equilibrium control $\lambda^*$ in the right hand side of the second condition in \eqref{NE_cond} is due to the role that it plays for the determination of $P_t=P^{\lambda,\lambda^*}_t$ also when the controller deviates from the equilibrium, cf. \eqref{P_def_strong_sol}. 
\end{remark}

\begin{remark} 
A connection between our equilibrium definition and a fixed-point in a suitable best response mapping can be established.    
In fact we will use this connection when proving the equilibrium existence result Theorem \ref{thm_eq_cand_ex}. 
Let $(\tau^*,\lambda^*)$ be any given admissible strategy pair. 
Then we may, in line our equilibrium definition, define the (point-to-set) \textit{best response mapping} of the stopper as
\begin{align*}
\lambda^* \in \mathbb{L} \mapsto \argmax_{\tau\in\mathbb{T}} \mathcal{J}^1\left(\tau,\lambda^*(P^{\lambda^*,\lambda^*}),p\right),
\end{align*}
while the  (point-to-set) \textit{best response mapping} of the controller is given by
\begin{align*}
(\tau^*,\lambda^*) \in \mathbb{T}\times\mathbb{L} \mapsto \argmax_{\lambda\in\mathbb{L}} 
\mathcal{J}^2\left(\tau^*,\lambda(P^{\lambda,\lambda^*}),p\right).
\end{align*}
It is then immediately clear that our equilibrium definition corresponds to a fixed-point in the best response mapping
\begin{align*}
(\tau^*,\lambda^*) \in \mathbb{T}\times\mathbb{L}   \mapsto 
\left(
\argmax_{\tau\in\mathbb{T}} \mathcal{J}^1\left(\tau,\lambda(P^{\lambda^*,\lambda^*}),p\right),
\argmax_{\lambda\in\mathbb{L}} \mathcal{J}^2\left(\tau^*,\lambda(P^{\lambda,\lambda^*}),p\right)
\right).
\end{align*}
\end{remark}

In the following result we conclude this section by establishing that $(P_t)$ does indeed correspond to the conditional probability of an 
active controller, i.e., $\{\theta=1\}$,  in case the controller does not deviate from an equilibrium (candidate).

\begin{proposition}\label{filter_prop}
Let $\lambda^* \in \mathbb{L}$ be an arbitrary admissible control. Suppose $\lambda^* = \lambda$ in \eqref{P_def_strong_sol} and consider a constant $0 < T<\infty$.  
Then the solution $P_t = P_t^{\lambda^*,\lambda^*}$ to \eqref{P_def_strong_sol} satisfies a.s., for each $0\leq t\leq T$, 
\begin{align*}
P_t=\E[\theta|\mathcal{F}_t^X]
\end{align*}
and 
\begin{align}\label{P-SDE-for-innovation}
dP_t=\lambda^*_t(P_t)P_t(1-P_t)d\hat{W}_t,\quad P_0=p,
\end{align}
where  
\begin{align*}
\hat{W}_t=X_t+ct-\int_0^t\lambda^*(P_s)P_sds,
\end{align*}
is a Brownian motion w.r.t. $((\mathcal{F}^X_t),\mathbb{P})$.
\end{proposition}
\begin{proof}
This proof is similar to that of \cite[Proposition 11]{ekstrom2022detect}. Define $\Pi_t=\mathbb{E}[\theta|\mathcal{F}_t^X]$. 
Then $\mathbb{E}[\theta\lambda^*(P_t)|\mathcal{F}_t^X]= \lambda^*(P_t)\Pi_t$, since $(P_t)$ is $(\mathcal{F}^X_t)$-adapted. 
Relying on standard filtering theory (see e.g., \cite[Chapter 8.1]{liptser2001statistics} and arguments similar to those in the proof of \cite[Proposition 11]{ekstrom2022detect}), 
it can now be seen, for $0\leq t\leq T$, that
\begin{align*}
d\Pi_t=\lambda^*(P_t)\Pi_t(1-\Pi_t)d\bar{W}_t,
\end{align*}
where
\begin{align*}
\bar{W}_t:=X_t+ct-\int_0^t\lambda^*(P_s)\Pi_sds
\end{align*}
is a Brownian motion with respect to $((\mathcal{F}^X_t),\mathbb{P})$. Hence, by the definition of $(X_t)$ in \eqref{control-process} it is directly seen that $(\Pi_t)$ satisfies the SDE
\begin{align*}
d\Pi_t=\lambda^*(P_t)\Pi_t(1-\Pi_t)(\theta\lambda^*(P_t)-\lambda^*(P_t)\Pi_t)dt+\lambda^*(P_t)\Pi_t(1-\Pi_t)dW_t, \enskip \Pi_0 = p.
\end{align*}
Recalling the definition of $(P_t)$ in \eqref{P_def_strong_sol}, we observe that $(P_t)$ and $(\Pi_t)$ are both strong solutions to the same SDE in case $\lambda^* = \lambda$. 
The results follow.
\end{proof}

\subsection{Searching for a threshold equilibrium}\label{sec:search-NE}
The aim of the present section is to search for an equilibrium of threshold type in the sense that the equilibrium strategy pair 
satisfies $(\tau^*,\lambda^*)=(\tau_{b_1^*},\lambda_{b_2^*})$ where 
\begin{align}
\tau_{b_1^*}:&=\inf\{t  \geq 0:P_t\leq b_1^*\}, \label{KL:wqf323avsf1-1}\\
p \rightarrow \lambda_{b_2^*}(p):&=\ubar{\lambda}+(\bar{\lambda}-\ubar{\lambda})I_{\{p< b^*_2\}},\label{KL:wqf323avsf1-2}
\end{align}
with $0<b_1^*<b_2^*<1$.

\begin{remark}\label{rem:threshold.inter}
The (double) threshold strategy pair defined by \eqref{KL:wqf323avsf1-1}--\eqref{KL:wqf323avsf1-2} corresponds to 
(i) stopping the first time that 
$(P_t)$---whose dynamics is in this case given by \eqref{P_def_strong_sol} with $\lambda(P_t) = \lambda^*(P_t)=\lambda_{b_2^*}(P_t)$---falls below $b^*_1$, and
(ii) the controller using the control process $(\lambda_{b_2^*}(P_t))$, which is equal to
the small controller rate $\ubar{\lambda}$ when $P_t\geq b_2^*$ and
the large controller rate $\bar{\lambda}$ when $P_t< b_2^*$. 
\end{remark}

We remark that the content of this section is mainly of motivational value and that a corresponding formal result is the verification theorem reported in Section \ref{sec:ver-thm}, below. 

\subsubsection{The perspective of the controller}\label{sec:controller-perspec}
Given a candidate equilibrium strategy $\lambda^*\in \mathbb{L}$ and supposing that the stopper uses a candidate equilibrium threshold strategy
of the kind \eqref{KL:wqf323avsf1-1} where $b_1^*\in (0,1)$, the controller faces the optimal control problem
\begin{align}\label{controller-value}
v(p,b_1^*):= \sup_{\lambda\in \mathbb{L}} \mathcal{J}^2\left(\tau_{b_1^*},\lambda(P^{\lambda,\lambda^*}_t),p\right),
\end{align} 
where we recall that $(P_t) = (P_t^{\lambda,\lambda^*})$ is given by \eqref{P_def_strong_sol}; however, 
due to the conditioning on $\theta = 1$ in the controller reward $\mathcal{J}^2$ (see \eqref{reward-employee}) we may here set $\theta = 1$ in \eqref{P_def_strong_sol}.

Indeed writing $v(p)=v(p,b_1^*)$ 
and relying on \eqref{reward-employee} with the underlying process $(P_t)$ in the representation \eqref{P_def_strong_sol} with $\theta=1$, 
we expect, using the usual dynamic programming arguments, that the optimal value $v(p)$ satisfies
\begin{align*}
\frac{\lambda^*(p)^2p^2(1-p)^2}{2}v_{pp}(p)
+\left(\lambda^*(p)p(1-p)\lambda -\lambda^*(p)^2p^2(1-p)\right) v_p(p)
-r v(p)+
c-(\lambda-\ubar{\lambda})^2\leq 0,
\end{align*}
for all $\lambda  \in \{\bar{\lambda},\ubar{\lambda}\}$ and $p \in (b_1^*,1)$, while equality should hold in case $\lambda^*(p)=\lambda$, i.e., 
\begin{align*}
\frac{\lambda^*(p)^2p^2(1-p)^2}{2}v_{pp}(p)
+\left(\lambda^*(p)p(1-p)\lambda^*(p) -\lambda^*(p)^2p^2(1-p)\right) v_p(p)
 \\ -r v(p) + c-(\lambda(p)^*-\ubar{\lambda})^2 = 0.
\end{align*}
We will from now on ease the presentation by sometimes writing e.g., $\lambda^*$ instead of $\lambda^*(p)$. 
By subtracting one of the two equations above from the other we obtain
\begin{align*}
(\lambda^*)^2p(1-p)v_p-(\lambda^*-\ubar{\lambda})^2-\lambda^*p(1-p)\lambda v_p+(\lambda-\ubar{\lambda})^2 \geq 0,
\end{align*}
which is equivalent to
\begin{align}\label{KL:1}
(\lambda^*-\lambda)\lambda^*p(1-p)v_p\geq (\lambda^*-\ubar{\lambda})^2-(\lambda-\ubar{\lambda})^2.
\end{align}
We conclude that if $(\tau_{b_1^*}, \lambda^*)$ is an equilibrium then $\lambda^*= \lambda^*(p)$ must satisfy \eqref{KL:1} 
for $\lambda  \in \{\bar{\lambda},\ubar{\lambda}\}$ and all $p \in ({b_1^*},1)$.

We now first consider the case $\lambda^*(p)=\ubar{\lambda}$ with the deviation $\lambda=\bar{\lambda}$ (if $\lambda=\ubar{\lambda}$, then \eqref{KL:1} trivially holds). 
In this case \eqref{KL:1}  becomes 
\begin{align*}
(\ubar{\lambda}-\bar{\lambda})\ubar{\lambda}p(1-p)v_p\geq-(\bar{\lambda}-\ubar{\lambda})^2,
\end{align*}
which is equivalent to
\begin{align}\label{eq-ubarlam}
p(1-p)v_p\leq \frac{\bar{\lambda}-\ubar{\lambda}}{\ubar{\lambda}}=:(A).
\end{align}
Supposing that $p(p-1)v_p$ is decreasing 
(this is under additional assumptions on the model parameters verified in Proposition \ref{dec_f}, below) 
we see, for any given equilibrium strategy $\lambda^*$, that if we can find a value for $p$ that gives equality in \eqref{eq-ubarlam}, then it is a lower threshold for the 
set of points $p$ where $\lambda^*(p)=\ubar{\lambda}$ is possible; 
i.e., for any $p$ smaller than this threshold we must have $\lambda^*(p)=\bar{\lambda}$.  
The interpretation is that if the stopper assigns a small probability to an active controller then the controller will control with the large rate $\bar{\lambda}$.

We now consider the case $\lambda^*(p)=\bar{\lambda}$ and obtain, similarly to the above, the condition
\begin{align*}
(\bar{\lambda}-\ubar{\lambda})\bar{\lambda}p(1-p)v_p\geq(\bar{\lambda}-\ubar{\lambda})^2,
\end{align*}
which in turns gives the condition
\begin{align}\label{eq-barlam}
p(1-p)v_p\geq \frac{\bar{\lambda}-\ubar{\lambda}}{\bar{\lambda}}=:(B).
\end{align}
Similarly to the analysis of (A) above, this gives us an upper threshold for $p$ where $\lambda^*(p)=\bar{\lambda}$ is possible; i.e., for any 
$p$ exceeding this threshold we need $\lambda^*(p) = \ubar{\lambda}$.

In order for \eqref{eq-ubarlam} and \eqref{eq-barlam} to be feasible conditions we need that $(A)$ minus $(B)$ is non-negative, which is is directly verified. Hence, with the observations above as a motivation we will search for an equilibrium strategy $\lambda^*$ of the threshold type \eqref{KL:wqf323avsf1-2}, 
where the threshold switching point $b_2^*$ is a such that
\begin{align}\label{p_hat_cond_NE}
\frac{\bar{\lambda}-\ubar{\lambda}}{\bar{\lambda}}\leq b_2^*(1-b_2^*)v_p(b_2^*)\leq \frac{\bar{\lambda}-\ubar{\lambda}}{\ubar{\lambda}}.
\end{align}
Note that \eqref{p_hat_cond_NE} indicates that there may be multiple Nash equilibria, since every $b_2^*$ satisfying \eqref{p_hat_cond_NE} results in an equilibrium candidate strategy for the controller. 
As our equilibrium controller candidate we will, however, consider a switching point $b_2^*$ that corresponds to equality in the right hand side inequality in \eqref{p_hat_cond_NE}. More precisely, we will search for an equilibrium controller strategy given by \eqref{KL:wqf323avsf1-2}, with $b_2^* \in(b_1^*,1)$ satisfying
\begin{align}\label{p_hat_def_eq}
 b_2^*(1-b_2^*)v_p(b_2^*)= \frac{\bar{\lambda}-\ubar{\lambda}}{\ubar{\lambda}}.
\end{align}
Let us lastly note that if the players use a threshold strategy pair 
$(b^*_1,b_2^*)$, defined as in \eqref{KL:wqf323avsf1-1}--\eqref{KL:wqf323avsf1-2}, with $0<b_1^*<b_2^*<1$, then it can be shown that the corresponding value for the controller, i.e.,

\begin{align*}
\mathcal{J}^2\left(\tau_{b_1^*},\lambda_{b_2^*}(P_t),p\right)&= \E\left[\int_0^{\tau_{b_1^*}} e^{-rs}\left(c-(\lambda_{b_2^*}(P_t)-\ubar{\lambda})^2\right)ds\Bigg|\theta=1\right],
\end{align*}
where  $(P_t)=\left(P^{\lambda_{b_2^*},\lambda_{b_2^*}}_t\right)$  is given by \eqref{P_def_strong_sol} with $\lambda(P_t)=\lambda^*(P_t)=\lambda_{b^*_2}(P_t)$, 
coincides with $v:=v(p,b_1^*,b_2^*)$  defined as the solution to 

\begin{align}\label{ODE_Employee}
\begin{split}
\frac{\bar{\lambda}^2p^2(1-p)^2}{2}v_{pp}(p)+\bar{\lambda}^2(1-p)^2pv_p(p)-rv(p)+c-(\bar{\lambda}-\ubar{\lambda})^2  &= 0,  \enskip p  \in (b_1^*,b_2^*), \\
\frac{\ubar{\lambda}^2p^2(1-p)^2}{2}v_{pp}(p)+\ubar{\lambda}^2(1-p)^2pv_p(p)-rv(p)+c &= 0, \enskip p  \in (b_2^*,b_1^*), \\
v({p})&=0, \enskip p \in [0, b_1^*],\\ 
\enskip v(1)&=\frac{c}{r},\\
v\in \mathcal{C}(0,1) \cap \mathcal{C}^1(b_1^*,1)\cap &\mathcal{C}^2((b_1^*,b_2^*)\cup(b_2^*,1)).
\end{split}
\end{align}
Indeed we will in the subsequent analysis show that we may choose a stopper-controller threshold pair $(b_1^*,b_2^*)$ which is an equilibrium 
with a controller value given by \eqref{ODE_Employee}, under certain parameter assumptions; see Theorems \ref{ver_theorem} and \ref{thm_eq_cand_ex}.

\begin{remark}\label{rem_bound_control} Note that \eqref{ODE_Employee} is a boundary value problem on $(b_1^*,1)$, whose solution $v$ has been extended to be equal to zero on $[0,b_1^*)$.  The boundary conditions of \eqref{ODE_Employee} follow immediately from the boundary cases $p \leq b_1^*$ and $p=1$, 
which result in immediate stopping (corresponding to no income for the controller) and never stopping (corresponding to the income rate $c$ earned forever), respectively.  
\end{remark}

\subsubsection{The perspective of the stopper}
If the players use a threshold strategy pair 
$(b^*_1,b_2^*)$, defined as in \eqref{KL:wqf323avsf1-1}--\eqref{KL:wqf323avsf1-2}, with $0<b_1^*<b_2^*<1$, then the corresponding value for the stopper is 
\begin{align*}
\mathcal{J}^1\left(\tau_{b_1^*},\lambda_{b_2^*}(P_t),p\right) = \E\left[\int_0^{\tau_{b_1^*}} e^{-rs}\left(\theta \lambda_{b_2^*}(P_t)-c\right)ds\right],
\end{align*}
where  $(P_t)=\left(P^{\lambda_{b_2^*},\lambda_{b_2^*}}_t\right)$  is given by \eqref{P_def_strong_sol} with $\lambda(P_t)=\lambda^*(P_t)=\lambda_{b^*_2}(P_t)$. However, since $\lambda=\lambda^*$ we may equivalently consider the dynamics of $(P_t)$ in the representation \eqref{P-SDE-for-innovation}; cf. Proposition \ref{filter_prop}.

Relying again on Proposition \ref{filter_prop} we may moreover use that $P_t=\E[\theta|\mathcal{F}_t^X]$ and iterated expectation to replace $\theta$ in the stopper reward  with $P_t$; in other words we  have the representation
\begin{align*}
\mathcal{J}^1\left(\tau_{b_1^*},\lambda_{b_2^*}(P_t),p\right) = \E\left[\int_0^{\tau_{b_1^*}} e^{-rs}\left(P_t \lambda_{b_2^*}(P_t)-c\right)ds\right],
\end{align*}
where $(P_t)$ is given by \eqref{P-SDE-for-innovation}. 
Based on this it can be shown that the stopper reward coincides with 
$u=u(p,b_1^*,b_2^*)$ defined as the solution to
\begin{align}\label{u_ODE}
\begin{split}
\frac{\bar{\lambda}^2p^2(1-p)^2}{2}u_{pp}(p)-ru(p)+p\bar{\lambda}-c  &=0, \enskip p \in (b^*_1,b^*_2),\\
\frac{\ubar{\lambda}^2p^2(1-p)^2}{2}u_{pp}(p)-ru(p)+p\ubar{\lambda}-c  &=0, \enskip p \in (b^*_2,1),\\
u(p)& =0, \enskip p \in [0,b^*_1],\\
u(1)&=\frac{\ubar{\lambda}-c}{r}, \\
u \in \mathcal{C}(0,1) \cap \mathcal{C}^1(b^*_1,1)\cap &\mathcal{C}^2((b^*_1,b_2^*)\cup(b_2^*,1)).
\end{split}
\end{align}
Note that \eqref{u_ODE} is also a boundary value problem on $(b^*_1,1)$ whose solution $u$ has been extended to be equal to zero on $[0,b^*_1)$. 
The boundary conditions of \eqref{u_ODE} can be interpreted using arguments similar to those in Remark \ref{rem_bound_control}. 

Lastly note that if $(b_1^*,b^*_2)$ corresponds to an equilibrium, then it should hold that 
$u_p(b_1^*,b_1^*,b_2^*)=0$, by the smooth fit principle of optimal stopping theory, which motivates condition \eqref{mainthmcond2} in Theorem \ref{ver_theorem} below.

\subsection{A threshold equilibrium verification theorem}\label{sec:ver-thm}

Here we present our first main result, which is a verification theorem based on the equilibrium conditions that were informally derived in Section \ref{sec:search-NE}.

\begin{theorem}[Verification] \label{ver_theorem} 
Let $b^*_1,b_2^*\in (0,1)$ satisfy $b_1^*<b_2^*$. Let $u(p)=u(p,b_1^*,b_2^*)$ and $v(p)=v(p,b_1^*,b_2^*)$ 
be solutions to the boundary value problems \eqref{ODE_Employee} and  \eqref{u_ODE}. 
Suppose that
\begin{align*}
u(p)&\geq 0,\quad p\in [0,1],\label{mainthmcond1} \tag{I}\\
u_p(b_1^*)&=0,\label{mainthmcond2} \tag{II}\\
\frac{d}{dp}(p(1-p)v_p(p))&<0,\quad p\in (b^*_1,b_2^*)\cup(b_2^*,1),\label{mainthmcond3} \tag{III}\\
b_2^*(1-b_2^*)v_p(b_2^*)&=\frac{\bar{\lambda}-\ubar{\lambda}}{\ubar{\lambda}}.\label{mainthmcond4} \tag{IV}
\end{align*}   
Then the stopper-controller strategy pair $(\tau_{b^*_1},\lambda_{b_2^*})\in \mathbb{T}\times\mathbb{L}$ corresponding to 
\begin{align}\label{KL:wqf323avsf}
\tau_{b^*_1}=\inf\{t  \geq :P_t\leq b^*_1\}
\enskip \text{and} \enskip 
\enskip p \rightarrow \lambda_{b_2^*}(p):=\ubar{\lambda}+(\bar{\lambda}-\ubar{\lambda})I_{\{p< b_2^*\}}
\end{align}
is a Nash equilibrium (Definition \ref{NE_cond}). Moreover, $u$ and $v$ correspond to the equilibrium values for the stopper and the controller respectively, i.e.,
\begin{align*}
u(p)=\mathcal{J}^1\left(\tau_{b^*_1},\lambda_{b_2^*}(P_t),p\right),
\quad 
v(p)=\mathcal{J}^2\left(\tau_{b^*_1},\lambda_{b_2^*}(P_t),p\right).
\end{align*}
\end{theorem}

\begin{remark} \label{parameters-conditions-forNE:rem}
(i) Recall that \eqref{KL:wqf323avsf} corresponds to the control process being $\lambda_{b_2^*}(P_t)$  
where $(P_t)$ is given by \eqref{P_def_strong_sol} with $\lambda(P_t)=\lambda^*(P_t)=\lambda_{b^*_2}(P_t)$. 
(ii) If a pair $(b_1^*,b_2^*)$ corresponds to an equilibrium as in Theorem \ref{ver_theorem} then the equilibrium values 
$u(p)=u(p,b_1^*,b_2^*)$ and $v(p)=v(p,b_1^*,b_2^*)$ can be determined explicitly by solving 
\eqref{ODE_Employee} and  \eqref{u_ODE}; cf. Section \ref{sec:proof-exist}. 
\end{remark}

\begin{proof} (of Theorem \ref{ver_theorem}.) 
For ease of exposition we write in this proof $\lambda^*=\lambda_{b^*_2}$ and $\tau_{b_1^*}=\tau_{b^*}$. 

\textbf{Optimality of $\tau_{b^*}$}. 
Note that \eqref{mainthmcond1} and \eqref{mainthmcond2}, together with the boundary condition $u(b_1^*)=0$, imply that $u_{pp}(b_1^*+)\geq 0$. Thus, using the ODE in \eqref{u_ODE}, we obtain
\begin{align}\label{parameters-conditions-forNE}
b_1^*\leq \frac{c}{\bar{\lambda}}.
\end{align}
Let $n$ be fixed number. 
Relying on Proposition \ref{filter_prop} which implies that $(P_t)$ solves \eqref{P-SDE-for-innovation}, as well as \eqref{mainthmcond2} and It\^{o}'s formula we obtain for an arbitrary stopping time $\tau$ that
\begin{align*}
e^{-r(\tau\wedge n)}u(P_{\tau \wedge n})&=u(p)  +\int_0^{\tau\wedge n}e^{-rt}\left(\frac{(\lambda^*(P_t)^2P_t^2(1-P_t)^2}{2}u_{pp}(P_t)-ru(P_t) \right)1_{\{P_t \notin \{b^*_1,b^*_2\}\}}dt\\
&\quad +\int_0^{\tau\wedge n}e^{-rt}\lambda^*(P_t)P_t(1-P_t)u_{p}(P_{t})d\hat{W}_t,
\end{align*}
where the It\^{o} integral is a martingale since the integrand is bounded. Now use \eqref{u_ODE} and \eqref{parameters-conditions-forNE} to see that
\begin{align}\label{KL:hejheh}
-\left(\frac{(\lambda^*)^2p^2(1-p)^2}{2}u_{pp}-ru\right) \geq \lambda^*p-c.
\end{align}
Using the above together with Proposition \ref{filter_prop} and iterated expectation, and \eqref{mainthmcond1}, we find that
\begin{align*}
u(p)&= \E\left[e^{-r(\tau\wedge n)}u(P_{\tau \wedge n})-\int_0^{\tau\wedge n}e^{-rt}\left(\frac{(\lambda^*(P_t)^2P_t^2(1-P_t)^2}{2}u_{pp}(P_t)-ru(P_t) \right)
1_{\{P_t \notin \{b^*_1,b^*_2\}\}}dt\right]\\
&\geq  \E\left[e^{-r(\tau\wedge n)}u(P_{\tau \wedge n})+\int_0^{\tau\wedge n}e^{-rt}\left(\lambda^*(P_t) P_t -c\right)dt\right]\\
&\geq  \E\left[\int_0^{\tau\wedge n}e^{-rt}\left(\lambda^*(P_t)\mathbb{E}[\theta|\mathcal{F}^X_t] -c\right)dt\right]\\
&=\E\left[\int_0^{\tau\wedge n}e^{-rt}\left(\theta\lambda^*(P_t)-c\right)dt\right].
\end{align*}
By sending $n\rightarrow \infty$ and relying on dominated convergence we thus obtain
\begin{align*}
u(p)\geq  \E\left[\int_0^{\tau}e^{-rt}\left(\theta\lambda^*(P_t)-c\right)dt\right]=\mathcal{J}^2\left(\tau,\lambda^*,p\right).
\end{align*}
Using similar arguments as above with $\tau = \tau_{b^*}$ we find, by observing that we have equality in \eqref{KL:hejheh} for $p\in (b^*_1,1)$, that 
\begin{align*}
u(p)&=\E\left[e^{-r(\tau_{b^*}\wedge n)}u(P_{\tau_{b^*} \wedge n})\right] + \E\left[\int_0^{\tau_{b^*}\wedge n}e^{-rt}\left(\theta\lambda^*(P_t) -c\right)dt\right].
\end{align*}
(Note that the equality above is trivial when $p \leq b^*_1$, since $u(0)=0$ for $p \leq b^*_1$). 

Using that $u$ is bounded together with $u(b^*_1)=0$ we find using dominated convergence 
that the first expectation above converges to zero as 
$\rightarrow \infty$. Hence, using dominated convergence again, we find that 
\begin{align*}
u(p)&=\E\left[\int_0^{\tau_{b^*}}e^{-rt}\left(\theta \lambda^*(P_t) -c\right)dt\right]= \mathcal{J}^{1}\left(\tau_{b^*}, \lambda^*,p\right).
\end{align*}
We conclude that 
\begin{align*}
u(p) = \mathcal{J}^{1}\left(\tau_{b^*}, \lambda^*,p\right) = \sup_{\tau \in \mathbb{T} }\mathcal{J}^{1}\left(\tau, \lambda^*,p\right).
\end{align*}

\textbf{Optimality of $\lambda^*$}.
The controller reward \eqref{reward-employee} is conditioned on $\theta= 1$. 
Hence, in order to find the optimal strategy for the controller, 
we consider the process $(P_t)$ defined by \eqref{P_def_strong_sol} with $\theta=1$; in particular, if the controller selects an admissible control $\lambda$, 
then $(P_t)$ is given by 
\begin{align*}
dP_t=\lambda^*(P_t)P_t(1-P_t)(\lambda(P_t)-\lambda^*(P_t)P_t)dt+\lambda^*(P_t)P_t(1-P_t)dW_t.
\end{align*}
We now define the process $(N_t)$ given by
\begin{align*}
N_t=e^{-r(t\wedge \tau_{b^*})}v(P_{t\wedge\tau_{b^*}})+\int_0^{t\wedge\tau_{b^*}}e^{-rs}(c-(\lambda(P_s)-\ubar{\lambda})^2)ds.
\end{align*}
Consider now an arbitrary admissible control strategy $\lambda \in \mathbb{L}$. Using It\^{o}'s formula we obtain for $t\leq \tau_{b^*}$ that
\begin{align*}
dN_t&=e^{-rt}\left(\frac{1}{2}(\lambda^*(P_t))^2P_t^2(1-P_t)^2v_{pp}(P_t)-rv(P_t)\right)1_{\{P_t \neq b^*_2\}}dt\\
&\quad +e^{-rt}\lambda^*(P_t)P_t(1-P_t)(\lambda(P_t)-\lambda^*(P_t)P_t)v_p(P_t) 1_{\{P_t \neq b^*_2\}}dt\\
&\quad +e^{-rt}(c-(\lambda(P_t)-\ubar{\lambda})^2) 1_{\{P_t \neq b^*_2\}}dt +e^{-rt}\lambda^*(P_t)P_t(1-P_t)v_p(P_t)dW_t.
\end{align*}
Hence, $(N_t)$ is an It\^o process with a drift coefficient given, for $p\in(b^*_1,b^*_2) \cup (b^*_2,1)$, by  
\begin{align*}
e^{-rt} \left(
\frac{1}{2}(\lambda^*(p))^2p^2(1-p)^2v_{pp}(p) -rv(p) + \lambda^*(p)p(1-p)(\lambda(p)-\lambda^*(p)p)v_p(p) + c-(\lambda(p)-\ubar{\lambda})^2\right). 
\end{align*}
Note that it also holds, for $p\in(b^*_1,b^*_2) \cup (b^*_2,1)$, that 
\begin{align*}
\frac{1}{2}(\lambda^*(p))^2p^2(1-p)^2v_{pp}(p)-rv(p) +\lambda^*(p)p(1-p)(\lambda^*(p)-\lambda^*(p)p)v_p(p) + c-(\lambda^*(p)-\ubar{\lambda})^2
=0.
\end{align*}
To see this use \eqref{ODE_Employee} and that $\lambda^*=\lambda_{b_2*}$ is given in \eqref{KL:wqf323avsf}. 

Multiplying the equation above by $e^{-rt}$ and subtracting the resulting left hand side (which is zero) from the drift coefficient of $(N_t)$ yields that the drift coefficient of $(N_t)$ can, for 
$p\in(b^*_1,b^*_2) \cup (b^*_2,1)$, be written as 
\begin{align*}
e^{-rt}
\left(
(\lambda(p)-\lambda^*(p))\lambda^*(p)p(1-p)v_p(p) + (\lambda^*(p)-\ubar{\lambda})^2 - (\lambda(p)-\ubar{\lambda})^2
\right).
\end{align*}
With arguments similar to those in Section \ref{sec:controller-perspec} we find that conditions \eqref{mainthmcond3} and \eqref{mainthmcond4} imply that 
the expression above is non-positive (compare the above expression with \eqref{KL:1}), i.e., the drift of $(N_t)$ is non-positive (regardless of the choice of $\lambda \in \mathbb{L}$).

We conclude that $(N_t)$ is a bounded process with non-positive drift. Using optional sampling we find
\begin{align*}
v(p)=N_0\geq\mathbb{E}[N_{\tau_{b^*}\wedge n}|\theta=1]
\end{align*}
for any $n\in\N$. Using dominated convergence and $\lim_{n\to \infty}e^{-r(\tau_{b^*}\wedge n)}v(P_{\tau_{b^*} \wedge n})=0$ a.s. we find
\begin{align*}
v(p)=N_0\geq\mathbb{E}\left[\lim_{n\to\infty}N_{\tau_{b^*}\wedge n}|\theta=1\right]
=\E\left[\int_0^{\tau_{b^*}}e^{-rt}(c-(\lambda(P_t)-\ubar{\lambda})^2)dt\Bigg|\theta=1\right]=
\mathcal{J}^{2}\left(\tau_{b^*},\lambda,p\right).
\end{align*}
Repeating the same arguments with $\lambda=\lambda^*$ we obtain that the drift of $(N_t)$ vanishes and that
\begin{align*}
v(p)=N_0=\E\left[\lim_{n\to\infty}N_{\tau_{b^*}\wedge n}|\theta=1\right]]=\E\left[\int_0^{\tau_{b^*}}e^{-rt}(c-(\lambda^*(P_t)-\ubar{\lambda}))^2dt\Bigg|\theta=1\right]=
\mathcal{J}^{2}\left(\tau_{b^*},\lambda^*,p\right).
\end{align*}
We conclude that
\begin{align*}
v(p)=\mathcal{J}^{2}\left(\tau_{b^*}, \lambda^*,p\right)=\sup_{\lambda \in \mathbb{L} }
\mathcal{J}^{2}\left(\tau_{b^*},\lambda,p\right).
\end{align*}
\end{proof}

\subsection{Equilibrium existence}\label{sec:existNE}
The main result of this section is Theorem \ref{thm_eq_cand_ex} which reports conditions on the primitives of the model that guarantee the existence of a threshold equilibrium. 
The proof of this result, which is reported in Section \ref{sec:proof-exist}, relies on the Poincaré-Miranda theorem and is in this sense a fixed-point type proof. In particular, the Poincaré-Miranda theorem follows from the Brouwer fixed-point theorem; cf. e.g., \cite{poincaremiranda2019}.

The following notation will be used throughout this section
\begin{align}\label{eq_alfa}
\alpha_1(\lambda):=\frac{1}{2} + \frac{\sqrt{8\frac{r}{\lambda^2}+1}}{2}, 
\enskip
\alpha_2(\lambda):=\frac{1}{2}-\frac{\sqrt{8\frac{r}{\lambda^2}+1}}{2}.
\end{align}
In particular, we will use these to express solutions to the ODEs in \eqref{ODE_Employee} and  \eqref{u_ODE}.

\begin{theorem}[Equilibrium existence]\label{thm_eq_cand_ex}
Suppose the model parameters $\ubar{\lambda},\bar{\lambda},c$ and $r$ are such that 
\begin{align}\label{cond_lemma_employee}
-\alpha_2(\bar{\lambda})\left(\frac{(\bar{\lambda}-\ubar{\lambda})^2}{r}-\frac{\bar{\lambda}-\ubar{\lambda}}{\alpha_1(\ubar{\lambda})\ubar{\lambda}}\right)<\frac{\bar{\lambda}-\ubar{\lambda}}{\ubar{\lambda}}<\frac{c}{r}
\end{align}
and 
\begin{align}
(\bar{\lambda}-\ubar{\lambda})^2\leq c \leq(1-\alpha_1(\ubar{\lambda}))\bar{\lambda}+\alpha_1(\ubar{\lambda})\ubar{\lambda}.\label{existence_thm_cond}
\end{align}
Then there exists constants $0<b^*_1<b^*_2<1$ 
such that the strategy pair $(\tau_{b^*_1},\lambda_{b^*_2})$ given by \eqref{KL:wqf323avsf} is a Nash equilibrium.
\end{theorem}

Figure \ref{fig:valueplot} contains a numerical example. 
 
\begin{remark} 
(i) The conditions \eqref{cond_lemma_employee}--\eqref{existence_thm_cond} of Theorem \ref{thm_eq_cand_ex} can be directly examined for any given parameter specification. 
(ii) If we set $\bar{\lambda}= \ubar{\lambda} + h$, then we can write these conditions as
\begin{align*}
-\frac{\alpha_2(\bar{\lambda})}{r}h^2+\frac{\alpha_2(\bar{\lambda})}{\alpha_1(\ubar{\lambda})\ubar{\lambda}}h
<\frac{h}{\ubar{\lambda}}<\frac{c}{r}
\end{align*}
and 
\begin{align*}
h^2&\leq c \leq(1-\alpha_1(\ubar{\lambda}))h + \ubar{\lambda}.
\end{align*}
Using this observation it is easily verified that there exists, for fixed $c$ and $r$,  a constant $\bar{h}\in(0,\infty)$ such that these conditions are satisfied for each $h\leq \bar{h}$. 
In other words, the conditions of Theorem \ref{thm_eq_cand_ex} hold, i.e., an equilibrium exists, whenever $\ubar{\lambda}$ and $\bar{\lambda}$ are sufficiently close to each other.
\end{remark}

\begin{figure}[h]
    \centering
   \includegraphics[scale=0.5]{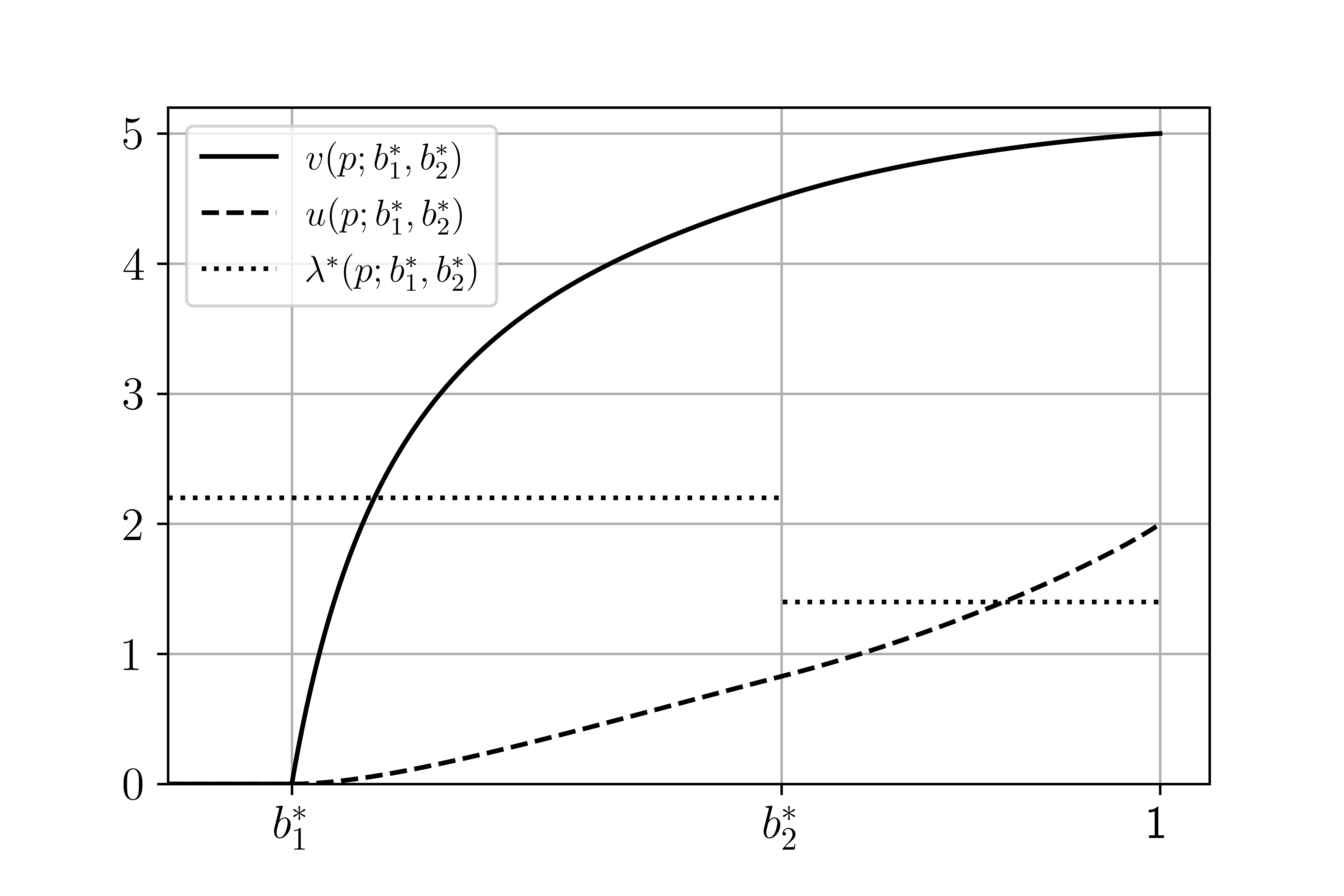}
    \caption{The value functions $v$ (controller) and $u$ (stopper), 
		as well as the equilibrium thresholds $(b_1^*\thickapprox 0.125,b_2^*\thickapprox 0.618)$ and the equilibrium controller strategy $\lambda_{b^*_2}$. 
		The parameters are $c=1,r=0.2,\bar{\lambda}=2.2$ and $\underline{\lambda}=1.4$.}
    \label{fig:valueplot}
\end{figure}

\subsection{The proof of Theorem \ref{thm_eq_cand_ex}}\label{sec:proof-exist}

The proof of Theorem \ref{thm_eq_cand_ex} is found in Section \ref{sec:the-proof-exist}. 
It relies on the content of Sections \ref{sec:obs-cont-thres}--\ref{sec:obs-stop-thres}.

\subsubsection{Observations regarding Equation \eqref{ODE_Employee}}\label{sec:obs-cont-thres}

Let $0<b_1^*<b_2^*<1$ be arbitrary constants. It can be verified that the solution $v(p)=v(p,b_1^*,b_2^*)$ to \eqref{ODE_Employee} is

\begin{align}\label{general-sol-v}
v(p,b_1^*,b_2^*)&=\begin{cases}
0,\quad p\leq b_1^*,\\
k_1\left(\frac{1-p}{p}\right)^{\alpha_1(\bar{\lambda})}+k_2\left(\frac{1-p}{p}\right)^{\alpha_2(\bar{\lambda})} + \frac{c-(\bar{\lambda}-\ubar{\lambda})^2}{r},\quad  b^*_1 < p< b^*_2,\\
k_3\left(\frac{1-p}{p}\right)^{\alpha_1(\ubar{\lambda})}+k_4\left(\frac{1-p}{p}\right)^{\alpha_2(\ubar{\lambda})} + \frac{c}{r},\quad p\geq b^*_2,
\end{cases}
\end{align}
where the constants $k_i,i=1,..,4$ can be determined by the boundary and smoothness conditions in \eqref{ODE_Employee}.
(Recall that $\alpha_i(\lambda),i=1,2$ are defined in \eqref{eq_alfa}.) 
However, instead of directly 
determining $k_i,i=1,..,4$ to attain these conditions we will determine these constants in order to attain only the boundary conditions and the continuity in 
\eqref{ODE_Employee} as well as the condition
\begin{align}\label{KL:aksldawrv}
v_p(b^*_2+,b_1^*,b_2^*)=\frac{\bar{\lambda}-\ubar{\lambda}}{b^*_2(1-b^*_2)\ubar{\lambda}}.
\end{align}
(The interpretation of \eqref{KL:aksldawrv} is that condition \eqref{mainthmcond4} in Theorem \ref{thm_eq_cand_ex} holds from the right.) After this we will show that $b^*_2$ can be chosen so that 
\eqref{KL:aksldawrv} also holds from the left (i.e., so that $v$ satisfies all conditions of \eqref{ODE_Employee} as well as \eqref{mainthmcond4}); 
see Lemma \ref{p_hat_ex_lemma} below.

First use that $v(1,b_1^*,b_2^*)=\frac{c}{r}$ implies that $k_4=0$. Note also that \eqref{KL:aksldawrv} implies that
\begin{align*}
k_3=-\frac{\bar{\lambda}-\ubar{\lambda}}{\alpha_1(\ubar{\lambda})\ubar{\lambda}}\left(\frac{1-b_2^*}{b_2^*}\right)^{-\alpha_1(\ubar{\lambda})}.
\end{align*}
Using these constants we obtain from \eqref{general-sol-v} that
\begin{align}\label{v_jump}
v(b^*_2+,b_1^*,b_2^*)=\frac{c}{r}-\frac{\bar{\lambda}-\ubar{\lambda}}{\alpha_1(\ubar{\lambda})\ubar{\lambda}}.
\end{align}
Using the condition $v(b_1^*,b_1^*,b_2^*)=0$ we obtain
\begin{align*}
k_2=\frac{(\bar{\lambda}-\ubar{\lambda})^2-c}{r}\left(\frac{1-b_1^*}{b_1^*}\right)^{-\alpha_2(\bar{\lambda})}-k_1\left(\frac{1-b_1^*}{b_1^*}\right)^{\alpha_1(\bar{\lambda})-\alpha_2(\bar{\lambda})}
\end{align*}
and hence, using also continuity $v(b^*_2-,b_1^*,b_2^*)=v(b^*_2+,b_1^*,b_2^*)$, we obtain
\begin{align*}
k_1=\frac{\frac{(\bar{\lambda}-\ubar{\lambda})^2}{r}-\frac{\bar{\lambda}-\ubar{\lambda}}{\alpha_1(\ubar{\lambda})\ubar{\lambda}}+\frac{c-(\bar{\lambda}-\ubar{\lambda})^2}{r}\left(\frac{1-b_1^*}{b_1^*}\right)^{-\alpha_2(\bar{\lambda})}\left(\frac{1-b^*_2}{b^*_2}\right)^{\alpha_2(\bar{\lambda})}}
{\left(\frac{1-b^*_2}{b^*_2}\right)^{\alpha_1(\bar{\lambda})}-\left(\frac{1-b_1^*}{b_1^*}\right)^{\alpha_1(\bar{\lambda})-\alpha_2(\bar{\lambda})}\left(\frac{1-b^*_2}{b^*_2}\right)^{\alpha_2(\bar{\lambda})}}.
\end{align*}
We need the following technical result in the proof of Theorem \ref{thm_eq_cand_ex} (in Section \ref{sec:the-proof-exist}). 
The proof can be found in Appendix \ref{app:lemmas}.

\begin{lemma}\label{p_hat_ex_lemma} 
Suppose \eqref{cond_lemma_employee} holds. (i) Let $v(p,b_1^*,b_2^*)$ be given by \eqref{general-sol-v}  with the constants $k_i,i=1,\dots,4$ determined above. 
Then 
\begin{align}\label{lim_p_to_1}
\lim_{b^*_2 \nearrow 1}b^*_2(1-b^*_2)v_p(b^*_2-,b_1^*,b_2^*)= 
-\alpha_2(\bar{\lambda})\left(\frac{(\bar{\lambda}-\ubar{\lambda})^2}{r}-\frac{\bar{\lambda}-\ubar{\lambda}}{\alpha_1(\ubar{\lambda})\ubar{\lambda}}\right)
\end{align}
and
\begin{equation}\label{lim_p_b}
\lim_{b^*_2\searrow b_1^*}b^*_2(1-b^*_2)v_p(b^*_2-,b_1^*,b_2^*)=\infty.
\end{equation} 
In particular,
\begin{align}\label{lemma_ex_p_eq}
\lim_{b^*_2\nearrow 1}b^*_2(1-b^*_2)v_p(b^*_2-,b_1^*,b_2^*)<\frac{\bar{\lambda}-\ubar{\lambda}}{\ubar{\lambda}}<\lim_{b^*_2\searrow b^*_1} b^*_2(1-b^*_2)v_p(b^*_2-,b_1^*,b_2^*).
\end{align}
(ii) For any fixed $b_1^*\in (0,1)$ there exists a $b^*_2\in(b_1^*,1)$ such that the solution 
$v(p)$ to \eqref{ODE_Employee} satisfies \eqref{mainthmcond4}.
\end{lemma}

\subsubsection{Observations regarding Equation  \eqref{u_ODE}}\label{sec:obs-stop-thres}

Let $0<b_1^*<b_2^*<1$ be arbitrary constants. It can be verified that the solution  $u(p)=u(p,b_1^*,b_2^*)$ to \eqref{u_ODE} is 
\begin{align}\label{value_employer}
u(p)=\begin{cases}
0, \quad p\leq b_1^*,\\
p\left(c_1\left(\frac{1-p}{p}\right)^{\alpha_1(\bar{\lambda})}+c_2\left(\frac{1-p}{p}\right)^{\alpha_2(\bar{\lambda})}\right)+\frac{p\bar{\lambda}-c}{r},\quad b^*_1 < p< b^*_2,\\
p\left(c_3\left(\frac{1-p}{p}\right)^{\alpha_1(\ubar{\lambda})}+c_4\left(\frac{1-p}{p}\right)^{\alpha_2(\ubar{\lambda})}\right)+\frac{p\ubar{\lambda}-c}{r},\quad p\geq b^*_2,
\end{cases}
\end{align}
where the constants $c_i$ are determined by the boundary and smoothness conditions in \eqref{u_ODE}. The boundary condition $u(1,b_1^*,b_2^*)=\frac{\ubar{\lambda}-c}{r}$ gives us $c_4=0$, while $u(b_1^*,b_1^*,b_2^*)=0$ gives us
\begin{align}\label{asfqwfvwfqwg}
c_2=\frac{c-b_1^*\bar{\lambda}}{b_1^*r}\left(\frac{1-b_1^*}{b_1^*}\right)^{-\alpha_2(\bar{\lambda})}-c_1\left(\frac{1-b_1^*}{b_1^*}\right)^{\alpha_1(\bar{\lambda})-\alpha_2(\bar{\lambda})}.
\end{align}
Finally, the remaining two conditions give us
\begin{align*}
c_3=\left(c_1\left(\frac{1-b^*_2}{b^*_2}\right)^{\alpha_1(\bar{\lambda})}+c_2\left(\frac{1-b^*_2}{b^*_2}\right)^{\alpha_2(\bar{\lambda})}+\frac{\bar{\lambda}-\ubar{\lambda}}{r}\right)
\left(\frac{1-b^*_2}{b^*_2}\right)^{-\alpha_1(\ubar{\lambda})}
\end{align*}
and
\begin{align}\label{c_1_eq}
c_1=\frac{\frac{\alpha_1(\ubar{\lambda})(\bar{\lambda}-\ubar{\lambda})}{r}\left(\frac{1-b^*_2}{b^*_2}\right)^{-\alpha_2(\bar{\lambda})}+(\alpha_2(\bar{\lambda})-\alpha_1(\ubar{\lambda}))\frac{b_1^*\bar{\lambda}-c}{b_1^*r}\left(\frac{1-b_1^*}{b_1^*}\right)^{-\alpha_2(\bar{\lambda})}}{(\alpha_1(\bar{\lambda})-\alpha_1(\ubar{\lambda}))\left(\frac{1-b^*_2}{b^*_2}\right)^{\alpha_1(\bar{\lambda})-\alpha_2(\bar{\lambda})}-(\alpha_2(\bar{\lambda})-\alpha_1(\ubar{\lambda}))\left(\frac{1-b_1^*}{b_1^*}\right)^{\alpha_1(\bar{\lambda})-\alpha_2(\bar{\lambda})}}.
\end{align}
We will make use of the following technical result in the proof of Theorem \ref{thm_eq_cand_ex}. The proof can be found in Appendix \ref{app:lemmas}. 

\begin{lemma}\label{lim_b_lemma}
Suppose \eqref{existence_thm_cond} holds. Then, for the solution to \eqref{u_ODE} it holds that 
\begin{align*}
\lim_{b_1^*\searrow 0} \sup_{b^*_2\in(b_1^*,1)}u_p(b_1^*+,b_1^*,b_2^*) = -\infty,
\enskip \text{and}\enskip 
\lim_{b_1^*\nearrow 1}\inf_{b^*_2\in(b_1^*,1)}u_p(b_1^*+,b_1^*,b_2^*) =\infty.
\end{align*}
\end{lemma}

\subsubsection{The proof}\label{sec:the-proof-exist}

\begin{proof}(Of Theorem \ref{thm_eq_cand_ex}.) 
The idea of the proof is to establish existence of a threshold strategy pair $(b_1^*,b_2^*)$ satisfying the conditions of Theorem \ref{ver_theorem}. 
The proof consists of several parts. Here  we establish existence of a pair $(b_1^*,b_2^*)$ with 
$0<b_1^*<b_2^*<1$ and corresponding functions 
$u(p)=u(p,b_1^*,b_2^*)$ and 
$v(p)=v(p,b_1^*,b_2^*)$ such that \eqref{ODE_Employee} and \eqref{u_ODE}, as well as \eqref{mainthmcond2} and \eqref{mainthmcond4} hold. 
The remaining conditions are established in Appendix \ref{app:ODE-exist}. In particular, \eqref{mainthmcond1} follows from Proposition \ref{u>0_prop} and 
\eqref{mainthmcond3} follows from Proposition \ref{dec_f}.

Consider a pair $(b_1^*,b^*_2)$ with $0<b_1^*<b^*_2<1$ and let $u(p,b_1^*,b^*_2)$ and $u(p,b_1^*,b^*_2)$ be given by \eqref{general-sol-v} and \eqref{value_employer} with the constants $c_i,k_i,i=1,...,4$
determined as in Sections \ref{sec:obs-cont-thres} and  \ref{sec:obs-stop-thres}. Then all we have left to do to is to show that the pair $(b_1^*,b_2^*)$ can be chosen so that
\begin{align}
u_p(b^*_1+,b^*_1,b^*_2)&=0, \label{smooth_cond}\\
b^*_2(1-b^*_2)v_p(b^*_2-,b^*_1,b^*_2))-\frac{\bar{\lambda}-\ubar{\lambda}}{\ubar{\lambda}}&=0. \label{after_smooth_cond}
\end{align}
To this end we introduce the notation 
\begin{align*}
f(b_1,b_2)&= u_p (b_1+,b_1,b_2),\\
g(b_1,b_2)&= \arctan\left(b_2(1-b_2)v_p(b_2-,b_1,b_2)-\frac{\bar{\lambda}-\ubar{\lambda}}{\ubar{\lambda}}\right),
\end{align*}
for an arbitrary threshold strategy pair $(b_1,b_2)\in A:= \{(x,y)\in \R^2:0<x<y<1\}$.

Using Lemma \ref{lim_b_lemma}, it is easy to see that there exist constants $0<\ubar{b}_1<\bar{b}_1<1$ such that:
(i) $f(\ubar{b}_1,b_2)<0$ for all $b_2>\ubar{b}_1$, and 
$f(\bar{b}_1,b_2)>0$ for all $b_2>\bar{b}_1$, and 
(ii) $f$ is continuous on the set $\tilde{A}:=A\cap ([\ubar{b}_1,\bar{b}_1]\times [0,1])$. 

Fix two such values $\ubar{b}_1$ and $\bar{b}_1$ (arbitrarily). We can now find a continuous extension of $f$ on the whole rectangle 
$[\ubar{b}_1,\bar{b}_1]\times [0,1]$ by
\begin{align*}
\tilde{f}(b_1,b_2)=\begin{cases}
f(b_1,b_2),\quad (b_1,b_2)\in \tilde A,\\
f(b_1,b_1),\quad (b_1,b_2)\not\in \tilde A.
\end{cases}
\end{align*}
We conclude that $\tilde{f}$ is continuous on $[\ubar{b}_1,\bar{b}_1]\times [0,1]$
with the properties that 
$f(\ubar{b}_1,b_2)<0$ for all $b_2 \in [0,1]$ 
and 
$f(\bar{b}_1,b_2)>0$ for all $b_2 \in [0,1]$.

Using Equation \eqref{lim_p_to_1} and \eqref{lim_p_b} we find a continuous extension of $g$ on the whole rectangle 
$[\ubar{b}_1,\bar{b}_1]\times [0,1]$ by
\begin{align}\label{asfjlkqwfuoveq}
\tilde{g}(b_1,b_2)=
\begin{cases}
g(b_1,b_2), \quad b_1<b_2,\\
\frac{\pi}{2},\quad\quad\quad\quad b_1\geq b_2.
\end{cases} 
\end{align}
Based on Lemma \ref{p_hat_ex_lemma} (in particular the left hand side inequality of \eqref{lemma_ex_p_eq}) we may now conclude that:
$\tilde{g}(b_1,1)<0$, for any $b_1 \in [\ubar{b}_1,\bar{b}_1]$ and 
$\tilde{g}(b_1,0)= \pi/2 >0$, for any $b_1 \in [\ubar{b}_1,\bar{b}_1]$.

The conclusions noted for $\tilde{f}$ and $\tilde{g}$ imply that 
we may use the Poincaré-Miranda theorem (cf. \cite{poincaremiranda2019}). 
In particular, it implies that there exists a pair $(b^*_1,b^*_2)\in [\ubar{b}_1,\bar{b}_1]\times [0,1]$ 
such that 
\begin{align}\label{asrlqrn}
\tilde{f}(b^*_1,b^*_2)=\tilde{g}(b^*_1,b^*_2)=0.
\end{align} 
Moreover, using \eqref{asfjlkqwfuoveq} and $\tilde{g}(b_1,1)<0$ (cf. above) we obtain $\tilde{g}(b_1,b_2)  \neq 0$ for $(b_1,b_2) \not\in \tilde A$, and hence 
$(b^*_1,b^*_2)\in \tilde A  \subset A$, i.e., $0<b^*_1<b^*_2<1$. 

It is now directly seen by the definitions of $f,\tilde f, g$ and $\tilde g$ that the pair $(b^*_1,b^*_2)$ satisfying \eqref{asrlqrn} is such that also 
\eqref{smooth_cond}--\eqref{after_smooth_cond} hold and we are done. 
\end{proof}

\section{Weak formulation}\label{sec:weak}
The purpose of this section is to consider a more general class of admissible control strategies compared to that of the strong formulation in Section \ref{Strong_sol_approach}. To this end we consider here a \textit{weak formulation} of our game based measure changes and Girsanov's theorem.  
We remark that this formulation is closely related to \cite{erik2022}, where a similar weak solution approach is used for an optimal control problem with discretionary stopping. The main finding of the present section is that the double threshold equilibrium of Theorem \ref{ver_theorem} is a Nash equilibrium also in the weak formulation. 

Let $(\Omega, \mathcal{A},\mathbb{P})$ be a probability space supporting a one-dimensional Brownian motion $(X_t)$ and a Bernoulli random variable $\theta$ with $\mathbb{P}(\theta=1)=p\in (0,1)$. 
Denote by $(\mathcal{F}_t^X)$ the smallest right continuous filtration to which $(X_t)$ is adapted.  
Define the terminal filtration according to $\mathcal{F}^X_\infty:=\sigma\left(\bigcup_{0\leq t\leq \infty}\mathcal{F}_t^X\right)$.
Define $(\mathcal{F}_t^{X,\theta})$ and $\mathcal{F}_\infty^{X,\theta}$  analogously. 

\begin{definition}[Admissibility in the weak formulation] \label{def:admissibility-weak} \enskip

\begin{itemize}
\item A process $(\lambda_t)$ is said to be an admissible control process if it has RCLL paths, is adapted to $(\mathcal{F}_t^{X,\theta})$, and takes values in 
$\{\bar{\lambda},\ubar{\lambda}\}$. The set of admissible control processes is denoted by $\tilde{\mathbb{L}}$. 

\item A stopping time $\tau$ is said to be an admissible stopping strategy if it is adapted to $(\mathcal{F}^X_t)$. 
The set of admissible stopping strategies is denoted by $\mathbb{T}$. 

\end{itemize}

\end{definition}

\begin{remark} The set of admissible stopping strategies in the weak formulation is analogous to set of admissible stopping strategies in the strong formulation. 
The main difference is instead that we define $(X_t)$ as a Brownian motion in the weak formulation, whereas $(X_t)$ is given by \eqref{control-process} in the strong formulation. 
\end{remark}

Now for any given control process $(\lambda_t)\in\tilde{\mathbb{L}}$ we define the process $(W^{\lambda}_t)$ according to 
\begin{align}\label{KL:X-in-weak}
X_t = \int_0^t(\theta\lambda_s-c)ds+W^{\lambda}_t.
\end{align}
By Girsanov's theorem (\cite[Chapter 3.5]{Karatzas2}) there exists a measure $\mathbb{P}_t^\lambda\sim \mathbb{P}$ on 
$(\Omega,(\mathcal{F}_t^{X,\theta}))$, given by 
\begin{align} \label{KL:RD-der}
\frac{d\mathbb{P}_t^{\lambda}}{d\mathbb{P}}\bigg|_{\mathcal{F}_t^{X,\theta}}= \exp\left(\int_0^t(\theta\lambda_t-c)dX_t-\frac{1}{2}\int_0^t(\theta\lambda_t-c)^2dt \right):=\Lambda^\lambda_t,
\end{align}
such that $\{W_t^\lambda,\mathcal{F}_t^{X,\theta};0\leq t\leq T\}$ is a Brownian motion on $(\Omega, \mathcal{F}^{X,\theta}_T,\mathbb{P}_T^\lambda)$ for each fixed $T\in [0,\infty)$. 
Moreover, we note that $(\Lambda^\lambda_t$) is a martingale by Novikov's condition. 
Thus, the theory of the Föllmer measure gives us the existence of a measure $\mathbb{P}^\lambda$ on $\mathcal{F}_\infty^{X,\theta}$, which satisfies $\mathbb{P}^\lambda(A)=\mathbb{P}_t^\lambda(A)$ for every $t\in[0,\infty)$ and $A\in\mathcal{F}^{X,\theta}_t$; see \cite[Section 2]{erik2022}, and also \cite{Fllmer1972TheEM} and \cite[p.192]{Karatzas2}.

This allows us to give definitions of the reward functions based on measure changes.

\begin{definition}\label{alternative_reward_def}
Given a strategy pair $\left(\tau,(\lambda_t)\right)\in \mathbb{T}\times\tilde{\mathbb{L}}$ we define the payoff of the stopper as
\begin{align}
\tilde{\mathcal{J}}^1\left(\tau,(\lambda_t),p\right)=\E^{\mathbb{P}^\lambda}\left[\int_0^\tau e^{-rs}\left(\theta\lambda_s-c\right)ds\right],
\end{align}
and the payoff of the controller as
\begin{align}
\tilde{\mathcal{J}}^2\left(\tau,(\lambda_t),p\right)=\E^{\mathbb{P}^\lambda}\left[\int_0^\tau e^{-rs}\left(c-(\lambda_s-\ubar{\lambda})^2\right)ds\Bigg|\theta=1\right].
\end{align}
\end{definition}

\begin{remark}\label{mot:weak-rewards} Let us motivate Definition \ref{alternative_reward_def} further. 
In the strong formulation (Section \ref{Strong_sol_approach}) we consider a fixed probability measure and define the controlled process $(X_t)$ in terms of a control process $(\lambda_t)$ and a given Brownian motion $(W_t)$; cf. \eqref{control-process}. 
In the present weak formulation we instead define $(X_t)$ as a Brownian motion, and let the control process $(\lambda_t)$ imply a measure change 
$\mathbb{P}^\lambda$, such that $W^\lambda$ defined by \eqref{KL:X-in-weak} is a Brownian motion under this measure. 
By comparing the resulting weak formulation equation for $(X_t)$ (i.e., \eqref{KL:X-in-weak}) and the equation for $(X_t)$ in the strong formulation (in \eqref{control-process}) 
the connection between the formulations becomes clear.  
\end{remark} 

The Nash equilibrium is now defined in the usual way: 

\begin{definition}\label{weak_sol_NE}
 A pair of admissible strategies $(\tau^*,(\lambda^*_t)) \in \mathbb{T} \times \mathbb{\tilde{L}}$ is a said to be a Nash equilibrium if 
\begin{align}\label{NE_cond2}
\begin{cases}
\tilde{\mathcal{J}}^1\left(\tau^*,(\lambda^*_t),p\right)\geq \tilde{\mathcal{J}}^1\left(\tau,(\lambda^*_t),p\right),\\
\tilde{\mathcal{J}}^2\left(\tau^*,(\lambda^*_t),p\right)\geq \tilde{\mathcal{J}}^2\left(\tau^*,(\lambda_t),p\right),
\end{cases}
\end{align}
for any pair of deviation strategies $(\tau,(\lambda_t)) \in \mathbb{T} \times \tilde{\mathbb{L}} $. 
\end{definition}
In line with the strong formulation solution approach (cf. \eqref{KL:wqf323avsf1-1} and \eqref{KL:wqf323avsf1-2}) we define a double threshold strategy pair $(\tau_{b_1^*},\lambda_{b_2^*})$, where $0<b_1^*<b_2^*<1$, by
\begin{align}
\tau_{b_1^*}:=\inf\{t\geq 0:P_t\leq b_1^*\}\label{tau_weak}\\
\lambda_{b_2^*}(P_t)=\ubar{\lambda}+(\bar{\lambda}-\ubar{\lambda})I_{\{P_t<b_2^*\}}\label{lambda_weak},
\end{align}
where $(P_t)$ is (in analogy with \eqref{XP_SDE}) given by the SDE
\begin{align}\label{P-equation}
dP_t=\lambda^*(P_t)P_t(1-P_t)(c-\lambda^*(P_t)P_t)dt+\lambda^*(P_t)P_t(1-P_t)dX_t, \enskip P_0=p.
\end{align}
(Recalling that $(X_t)$ is a Brownian motion we find that \eqref{P-equation} has a strong solution using analogues arguments as in the strong solution formulation; cf. Proposition \ref{prop_strong_ex}). 

The main result of the present section is that the double threshold equilibrium investigated in the strong formulation is also an equilibrium in the  weak formulation. Note that this implies that equilibrium existence in the weak formulation is guaranteed by the same parameter conditions as in Theorem \ref{thm_eq_cand_ex}.

\begin{theorem}
Suppose $(b^*_1,b^*_2)$ with $0<b^*_1<b^*_2<1$ are such that the conditions of Theorem \ref{ver_theorem} hold (implying that they correspond to a double threshold equilibrium 
\eqref{KL:wqf323avsf1-1}--\eqref{KL:wqf323avsf1-2}, in the strong formulation). Then, $(\tau_{b^*_1},\lambda_{b^*_2})$ given by 
\eqref{tau_weak}, \eqref{lambda_weak} and \eqref{P-equation} is an equilibrium in the weak formulation (Definition \ref{weak_sol_NE}). 
\end{theorem}
\begin{proof} In this proof we write $\lambda^*=\lambda_{b_2^*}$. For any admissible deviation strategy $(\lambda_t)$ it follows from \eqref{KL:X-in-weak} and \eqref{P-equation} that 
$(P_t)$ has the representation 
\begin{align}\label{KL:P-def-deviation-weak}
dP_t=\lambda^*(P_t)P_t(1-P_t)(\theta\lambda_t-\lambda^*(P_t)P_t)dt+\lambda^*(P_t)P_t(1-P_t)dW^\lambda_t,
\end{align}
where we recall that $(W^\lambda_t)$ is a Brownian motion under the measure $\mathbb{P}^\lambda$. We remark that $(P_t)$ depends in this sense on both $\lambda^*$ and $(\lambda_t)$ when the controller deviates from the equilibrium.

Note that the representation of  $(P_t)=(P_t^{\lambda^*,\lambda^*})$ in \eqref{KL:P-def-deviation-weak} is analogous to \eqref{P_def_strong_sol} in the strong formulation. Moreover, it is directly seen that the value functions are the same for both formulations in the case of no deviation, i.e.,
\begin{align*}
 \mathcal{J}^{1}\left(\tau^*,\lambda^*(P^{\lambda^*,\lambda^*}),p\right)=\mathcal{\tilde{J}}^1\left(\tau^*,\lambda^*,p\right),\\
 \mathcal{J}^{2}\left(\tau^*,\lambda^*(P^{\lambda^*,\lambda^*}),p\right)=\mathcal{\tilde{J}}^2\left(\tau^*,\lambda^*,p\right).
\end{align*}
Hence, 
$u(p)=\mathcal{\tilde{J}}^1\left(\tau^*,\lambda^*,p\right)$ and
$v(p)=\mathcal{\tilde{J}}^2\left(\tau^*,\lambda^*,p\right)$, with $u(p)$ and $v(p)$ as in Theorem \ref{ver_theorem}. 
Using this it is directly checked that the proof of Theorem \ref{ver_theorem} can be adjusted so that it shows that $(b_1^*,b_2^*)$ corresponds to an equilibrium also in the present weak formulation.  Indeed, this requires only minor adjustments including that $(P_t)$ is here given by \eqref{KL:P-def-deviation-weak}, 
and that the deviation strategies are allowed to be processes in $\tilde{\mathbb{L}}$. 
Particularly, note that Proposition \ref{filter_prop} holds also in this case. 
\end{proof}

\noindent \textbf{Acknowledgment} The authors are grateful to Erik Ekström at Uppsala University for discussions regarding games of the kind studied in the present paper, and suggestions that lead to improvements of this manuscript.

\appendix

\section{Properties of $(P_t)$ in the strong formulation} \label{app:SDE}
\begin{proposition}\label{prop_strong_ex}
 The SDE \eqref{P_def_strong_sol} has a strong solution $(P_t) = \left(P^{\lambda,\lambda^*}_t\right)$ for any admissible pair of Markov strategies $(\lambda^*,\lambda) \in \mathbb{L}^2$. 
\end{proposition}
\begin{proof}
Consider the interval $I=[\epsilon,1-\epsilon]$, for a small arbitrary  constant $\epsilon>0$. Then the diffusion coefficient is uniformly bounded away from zero in $I$. Thus we obtain for both cases $\{\theta=0\}$ and $\{\theta=1\}$ that:
(i) a weak solution $(P_t)$ to \eqref{P_def_strong_sol} exists (cf. e.g., \cite[Ch. 5]{Karatzas2}), and 
(ii) a solution to \eqref{P_def_strong_sol} is pathwise unique in $I$ (see \cite{Nakao1972}). By Lemma \ref{lem_bound_cond}, we obtain that $(P_t)$ cannot reach $0$ or $1$ in finite time. Hence, \eqref{P_def_strong_sol} admits a strong solution $(P_t)$ by \cite[Corollary 3.23]{Karatzas2}. 
\end{proof}

\begin{lemma}\label{lem_bound_cond}
For any pair $(\lambda^*,\lambda)\in\mathbb{L}^2$ it holds for $(P_t) = \left(P^{\lambda,\lambda^*}_t\right)$ given by \eqref{P_def_strong_sol} that
\begin{align}
\tau_0&:=\inf\{t\geq 0:P_t=0\}=\infty, \label{tau_0_infty}\\
\tau_1&:=\inf\{t\geq 0:P_t=1\}=\infty. \label{tau_1_infty}
\end{align}
\end{lemma}

\begin{proof}
In order to prove \eqref{tau_0_infty}, it is suffices to show that
\begin{align*}
\tilde{\tau}_0&:=\inf\{t\geq 0:\tilde{P}_t=0\} = \infty,
\end{align*}
where $\tilde{P}_t$ solves the SDE
\begin{align}\label{tilde_SDE}
d\tilde{P}_t=-\bar{\lambda}^2\tilde{P}^2_t(1-\tilde{P}_t)dt+\lambda^*(\tilde{P}_t)\tilde{P}_t(1-\tilde{P}_t)dW_t,
\end{align}
with $\tilde{P_0}=P_0=p$; indeed, it follows by comparison (see \cite[Chapter IX.3]{revuz2013continuous}) that $\tilde{\tau}_0\leq \tau_0 $ a.s. 
(The existence of a strong solution to \eqref{tilde_SDE} is given by arguments similar to those in the proof of Proposition \ref{prop_strong_ex}.)

Since $\lambda^*$ is RCLL in $[0,1]$ (and is piece-wise constant), there exists a $z\in (0,1)$ such that $\lambda^*(p)=\lambda^*(z)$ for $p\in (0,z]$. 
With some calculations we now obtain that the scale function of \eqref{tilde_SDE}, is for $a \leq z$ given by
\begin{align*}
s'(a)
&=\exp\left(-2\int_z^a\frac{-\bar{\lambda}^2p^2(1-p)}{(\lambda^*(p))^2p^2(1-p)^2}dp\right)\\
&= \exp\left(-2\left(\frac{\bar{\lambda}}{\lambda^*(z)}\right)^2\int^z_a\frac{1}{(1-p)}dp\right)\\
&=C(z)\frac{1}{(1-a)^{2\left(\frac{\bar{\lambda}}{\lambda^*(z)}\right)^2}},
\end{align*} 
where $C(z)>0$, and density of the speed measure for $p\leq z$ is given by 
\begin{align*}
m(p)=\frac{2}{\lambda^*(z)^2p^2(1-p)^2s'(p)}.
\end{align*}
Using that $s'(x)$ is increasing for $x\in(0,1)$ we have
\begin{align*}
\int_0^zs'(a)\left(\int_a^zm(p)dp\right)da& \geq  C(z) \frac{2}{\lambda^*(z)^2(1-z)^2s'(z)}\int_0^z\left(\int_a^z\frac{1}{p^2}dp\right)da = \infty.
\end{align*}
Hence, $\tilde{\tau}_0=\infty$ follows from Feller's test for explosion, and 
\eqref{tau_0_infty} follows. 


For reasons similar to the proof of the previous statement it is sufficient to prove that 
$\hat{\tau}_1:=\inf\{t\geq 0:\hat{P}_t=1\} = \infty$
where
\begin{align}
d\hat{P}_t=\bar{\lambda}^2\hat{P}_t(1-\hat{P}_t)dt+\lambda^*(\hat{P}_t)\hat{P}_t(1-\hat{P}_t)dW_t.\label{hat_SDE}
\end{align}
We fix a $z\in (0,1)$ such that $\lambda^*(p)=\lambda^*(z)$ for $p \in [z,1)$. For \eqref{hat_SDE} and $ a \in (z,1)$, we have
\begin{align*}
s'(a)& =\exp\left(-2\int_z^a\frac{\bar{\lambda}^2p(1-p)}{(\lambda^*(p))^2p^2(1-p)^2}dp\right)\\
& =D_1(z)\left(\frac{1-a}{a}\right)^{2\left(\frac{\bar{\lambda}}{\lambda^*(z)}\right)^2},
\end{align*}
where $D_1(z)>0$, and for $p \in (z,1)$, we have 
\begin{align*}
m(p)=\frac{2}{\lambda^*(z)^2p^2(1-p)^2s'(p)}.
\end{align*}
Hence, for some positive constants $D_2(z),D_3(z),D_4(z)$, we have 
\begin{align*}
\int^1_zs'(a)\left(\int_z^am(p)dp\right)da&\geq D_2(z)\int_z^1(1-a)^{2\left(\frac{\bar{\lambda}}{\lambda^*(z)}\right)^2}\left(\int^a_z\frac{1}{(1-p)^{2\left(\frac{\bar{\lambda}}{\lambda^*(z)}\right)^2+2}}dp\right)da\\
&=D_3(z)\int_z^1\frac{1}{1-a}da-D_4(z)=\infty.
\end{align*}
Hence, $\hat{\tau}_1=\infty$ follows by Feller's test for explosion, and \eqref{tau_1_infty} follows.

\end{proof}

\section{Proofs of Lemmas \ref{p_hat_ex_lemma} and \ref{lim_b_lemma}}\label{app:lemmas}

\begin{proof} (of Lemma \ref{p_hat_ex_lemma}.)
Observe that 
\begin{align*}
b^*_2(1-b^*_2)v_p(b^*_2-,b_1^*,b_2^*)&=-k_1\left(\alpha_1(\bar{\lambda})\left(\frac{1-b^*_2}{b^*_2}\right)^{\alpha_1(\bar{\lambda})}-\alpha_2(\bar{\lambda})\left(\frac{1-b_1^*}{b_1^*}\right)^{\alpha_1(\bar{\lambda})-\alpha_2(\bar{\lambda})}\left(\frac{1-b^*_2}{b^*_2}\right)^{\alpha_2(\bar{\lambda})}\right)\\
&\enskip+\alpha_2(\bar{\lambda})\frac{c-(\bar{\lambda}-\ubar{\lambda})^2}{r}\left(\frac{1-b_1^*}{b_1^*}\right)^{-\alpha_2(\bar{\lambda})}\left(\frac{1-b^*_2}{b^*_2}\right)^{\alpha_2(\bar{\lambda})}.
\end{align*}
First we consider the limit $b^*_2\to 1$. For the first part we have that
\begin{align*}
&k_1\alpha_1(\bar{\lambda})\left(\frac{1-b^*_2}{b^*_2}\right)^{\alpha_1(\bar{\lambda})}\\
&\enskip=\frac{\left(\frac{(\bar{\lambda}-\ubar{\lambda})^2}{r}-\frac{\bar{\lambda}-\ubar{\lambda}}{\alpha_1(\ubar{\lambda})\ubar{\lambda}}\right)\left(\frac{1-b^*_2}{b^*_2}\right)^{-\alpha_2(\bar{\lambda})}+\frac{c-(\bar{\lambda}-\ubar{\lambda})^2}{r}\left(\frac{1-b_1^*}{b_1^*}\right)^{-\alpha_2(\bar{\lambda})}}
{\left(\frac{1-b^*_2}{b^*_2}\right)^{\alpha_1(\bar{\lambda})-\alpha_2(\bar{\lambda})}-\left(\frac{1-b_1^*}{b_1^*}\right)^{\alpha_1(\bar{\lambda})-\alpha_2(\bar{\lambda})}}\alpha_1(\bar{\lambda})\left(\frac{1-b^*_2}{b^*_2}\right)^{\alpha_1(\bar{\lambda})}\rightarrow 0,
\end{align*}
as $b^*_2\nearrow 1$, since $\alpha_1(\bar{\lambda}),-\alpha_2(\bar{\lambda})>0$.
We consider the remaining term. We obtain
\begin{align*}
&k_1\alpha_2(\bar{\lambda})\left(\frac{1-b_1^*}{b_1^*}\right)^{\alpha_1(\bar{\lambda})-\alpha_2(\bar{\lambda})}\left(\frac{1-b^*_2}{b^*_2}\right)^{\alpha_2(\bar{\lambda})}+\alpha_2(\bar{\lambda})\frac{c-(\bar{\lambda}-\ubar{\lambda})^2}{r}\left(\frac{1-b_1^*}{b_1^*}\right)^{-\alpha_2(\bar{\lambda})}\left(\frac{1-b^*_2}{b^*_2}\right)^{\alpha_2(\bar{\lambda})}\\
&\enskip=\alpha_2(\bar{\lambda})\frac{\frac{(\bar{\lambda}-\ubar{\lambda})^2}{r}-\frac{\bar{\lambda}-\ubar{\lambda}}{\alpha_1(\ubar{\lambda})\ubar{\lambda}}}
{\left(\frac{1-b^*_2}{b^*_2}\right)^{\alpha_1(\bar{\lambda})-\alpha_2(\bar{\lambda})}\left(\frac{1-b_1^*}{b_1^*}\right)^{\alpha_2(\bar{\lambda})-\alpha_1(\bar{\lambda})}-1}\\&\quad+\alpha_2(\bar{\lambda})\left(\frac{1-b^*_2}{b^*_2}\right)^{\alpha_2(\bar{\lambda})}\left(
\frac{\frac{c-(\bar{\lambda}-\ubar{\lambda})^2}{r}\left(\frac{1-b_1^*}{b_1^*}\right)^{-\alpha_2(\bar{\lambda})}}
{\left(\frac{1-b^*_2}{b^*_2}\right)^{\alpha_1(\bar{\lambda})-\alpha_2(\bar{\lambda})}\left(\frac{1-b_1^*}{b_1^*}\right)^{\alpha_2(\bar{\lambda})-\alpha_1(\bar{\lambda})}-1}+\frac{c-(\bar{\lambda}-\ubar{\lambda})^2}{r}\left(\frac{1-b_1^*}{b_1^*}\right)^{-\alpha_2(\bar{\lambda})}\right)\\
&\enskip= (A1)+(B1).
\end{align*}
For $(A1)$ we obtain using $\alpha_1(\bar{\lambda})-\alpha_2(\bar{\lambda})>0$, that
\begin{align*}
(A1)=\alpha_2(\bar{\lambda})\frac{\frac{(\bar{\lambda}-\ubar{\lambda})^2}{r}-\frac{\bar{\lambda}-\ubar{\lambda}}{\alpha_1(\ubar{\lambda})\ubar{\lambda}}}
{\left(\frac{1-b^*_2}{b^*_2}\right)^{\alpha_1(\bar{\lambda})-\alpha_2(\bar{\lambda})}\left(\frac{1-b_1^*}{b_1^*}\right)^{\alpha_2(\bar{\lambda})-\alpha_1(\bar{\lambda})}-1}\rightarrow -\alpha_2(\bar{\lambda})\left(\frac{(\bar{\lambda}-\ubar{\lambda})^2}{r}-\frac{\bar{\lambda}-\ubar{\lambda}}{\alpha_1(\ubar{\lambda})\ubar{\lambda}}\right),
\end{align*}
as $b^*_2 \nearrow 1$. For $(B1)$ we have
\begin{align*}
(B1)&=\alpha_2(\bar{\lambda})\frac{c-(\bar{\lambda}-\ubar{\lambda})^2}{r}\left(\frac{1-b_1^*}{b_1^*}\right)^{-\alpha_2(\bar{\lambda})}\left(
\frac{\left(\frac{1-b^*_2}{b^*_2}\right)^{\alpha_1(\bar{\lambda})}\left(\frac{1-b_1^*}{b_1^*}\right)^{\alpha_2(\bar{\lambda})-\alpha_1(\bar{\lambda})}}
{\left(\frac{1-b^*_2}{b^*_2}\right)^{\alpha_1(\bar{\lambda})-\alpha_2(\bar{\lambda})}\left(\frac{1-b_1^*}{b_1^*}\right)^{\alpha_2(\bar{\lambda})-\alpha_1(\bar{\lambda})}-1}\right)\\
&\rightarrow 0,
\end{align*}
as $b^*_2 \nearrow 1$. Adding the limits gives us \eqref{lim_p_to_1}, and using \eqref{cond_lemma_employee} we thus obtain the first part of \eqref{lemma_ex_p_eq}.

For the second limit we find
\begin{align}\label{eq_limit_b}
\begin{split}
& \lim_{b^*_2\to b_1^*}b^*_2(1-b^*_2)v_p(b^*_2-,b_1^*,b_2^*)\\
& =\alpha_2(\bar{\lambda})\frac{c-(\bar{\lambda}-\ubar{\lambda})^2}{r}+\left(\frac{1-b_1^*}{b_1^*}\right)^{\alpha_1(\bar{\lambda})}(\alpha_1(\bar{\lambda})-\alpha_2(\bar{\lambda})) \lim_{b^*_2\to b_1^*}-k_1.
\end{split}
\end{align}
We note that $\alpha_1(\bar{\lambda})-\alpha_2(\bar{\lambda})>0$ and further investigate the limit by considering the denominator and numerator of $k_1$ separately. For the denominator of $k_1$ we have that
\begin{align*}
&\left(\frac{1-b^*_2}{b^*_2}\right)^{\alpha_1(\bar{\lambda})}-\left(\frac{1-b_1^*}{b_1^*}\right)^{\alpha_1(\bar{\lambda})-\alpha_2(\bar{\lambda})}\left(\frac{1-b^*_2}{b^*_2}\right)^{\alpha_2(\bar{\lambda})}\\
&\enskip=\left(\frac{1-b^*_2}{b^*_2}\right)^{\alpha_1(\bar{\lambda})}\left(1-\left(\frac{1-b_1^*}{b_1^*}\right)^{\alpha_1(\bar{\lambda})-\alpha_2(\bar{\lambda})}\left(\frac{1-b^*_2}{b^*_2}\right)^{\alpha_2(\bar{\lambda})-\alpha_1(\bar{\lambda})}\right)\nearrow 0,
\end{align*}
as $b^*_2\searrow b_1^*$; to see this use e.g., that $ x \mapsto \frac{1-x}{x}$ is decreasing for $x>0$. For the numerator of  $k_1$ we use \eqref{cond_lemma_employee} to find
\begin{align*}
&\frac{(\bar{\lambda}-\ubar{\lambda})^2}{r}-\frac{\bar{\lambda}-\ubar{\lambda}}{\alpha_1(\ubar{\lambda})\ubar{\lambda}}+\frac{c-(\bar{\lambda}-\ubar{\lambda})^2}{r}\left(\frac{1-b_1^*}{b_1^*}\right)^{-\alpha_2(\bar{\lambda})}\left(\frac{1-b^*_2}{b^*_2}\right)^{\alpha_2(\bar{\lambda})}\\
&\enskip\rightarrow \frac{c}{r}-\frac{\bar{\lambda}-\ubar{\lambda}}{\alpha_1(\ubar{\lambda})\ubar{\lambda}}>0.
\end{align*}
It follows that $\lim_{b^*_2\searrow b_1^*}k_1=-\infty$ and by \eqref{eq_limit_b} we obtain \eqref{lim_p_b} (from which the second part of \eqref{lemma_ex_p_eq} follows). Hence, statement (i) has been proved. 

Relying on the continuity of $v_p(b_2^*-,b_1^*,b_2^*)$ for $b_2^*\in (b_1^*,1)$ it follows immediately from \eqref{lemma_ex_p_eq} and the intermediate value theorem that we can choose $b^*_2$ so that 
$$v_p(b^*_2-,b_1^*,b_2^*)=\frac{\bar{\lambda}-\ubar{\lambda}}{b^*_2(1-b^*_2)\ubar{\lambda}}.$$
Hence, if we choose $b_2^*$ in this way then $p \mapsto v(p,b_1^*,b_2^*)$ satisfies \eqref{ODE_Employee} as well as \eqref{mainthmcond4} and hence statement (ii) holds.  
\end{proof} 

We need the following technical result  in the proof of  Lemma \ref{lim_b_lemma}. 
 
\begin{lemma}\label{c_1 lemma}
Let $c_1$ be given by \eqref{c_1_eq}, then we have
\begin{enumerate}[(a)]
\item $$
c_1 >-\left(\frac{1-b_1^*}{b_1^*}\right)^{-\alpha_1(\bar{\lambda})}\frac{b_1^*\bar{\lambda}-c}{b_1^*r},
$$
for $b_1^*\leq c/\bar{\lambda}$,
\item
$$
c_1 < \left(\frac{\alpha_1(\ubar{\lambda})(\bar{\lambda}-\ubar{\lambda})}{r(\alpha_1(\bar{\lambda})-\alpha_2(\bar{\lambda}))}-\frac{b_1^*\bar{\lambda}-c}{b_1^*r}\right)\left(\frac{1-b_1^*}{b_1^*}\right)^{-\alpha_1(\bar{\lambda})},
$$

for $b_1^* \geq c/\bar{\lambda}$.
\end{enumerate}
\end{lemma}
\begin{proof}
Let us prove (a) by showing that showing that $c_1$ is strictly decreasing in $b^*_2$ (recall that $b^*_2\in (b_1^*,1)$); the result then follows by taking $b^*_2=1$ in $c_1$.  
It holds that the denominator of $c_1$ is positive and strictly increasing in $b^*_2$. To see this use e.g., that $\alpha_1(\bar{\lambda})-\alpha_1(\ubar{\lambda})<0$ and that $\frac{1-x}{x}$ is strictly decreasing for $x>0$, which implies that
\begin{align*}
&(\alpha_1(\bar{\lambda})-\alpha_1(\ubar{\lambda}))\left(\frac{1-b^*_2}{b^*_2}\right)^{\alpha_1(\bar{\lambda})-\alpha_2(\bar{\lambda})}-(\alpha_2(\bar{\lambda})-\alpha_1(\ubar{\lambda}))\left(\frac{1-b_1^*}{b_1^*}\right)^{\alpha_1(\bar{\lambda})-\alpha_2(\bar{\lambda})}\\&\enskip\geq \left(\frac{1-b_1^*}{b_1^*}\right)^{\alpha_1(\bar{\lambda})-\alpha_2(\bar{\lambda})}(\alpha_1(\bar{\lambda})-\alpha_2(\bar{\lambda}))>0.
\end{align*}
Additionally, $-\alpha_2(\bar{\lambda})>0$ and $b_1^*\leq c/\bar{\lambda}$ implies that the numerator is positive and decreasing in $b^*_2$. We conclude that $c_1$ is strictly decreasing on $b_2^*$.

In order to prove (b) we write 
\begin{align*}
c_1=\frac{\frac{\alpha_1(\ubar{\lambda})(\bar{\lambda}-\ubar{\lambda})}{r}\left(\frac{1-b_2^*}{b_2^*}\right)^{-\alpha_2(\bar{\lambda})}}{D(b_1^*,b_2^*)}+\frac{(\alpha_2(\bar{\lambda})-\alpha_1(\ubar{\lambda}))\frac{b_1^*\bar{\lambda}-c}{b_1^*r}\left(\frac{1-b_1^*}{b_1^*}\right)^{-\alpha_2(\bar{\lambda})}}{D(b_1^*,b_2^*)}
\end{align*}
where $D(b_1^*,b_2^*)$ denotes the denominator of $c_1$. Note that $b_1^* \geq c/\bar{\lambda}$ implies that the second expression is non-positive. Thus, by similar arguments to (a) the result follows by taking $b_2^*=b_1^*$ in the first and $b_2^*=1$ in the second expression.
\end{proof}

\begin{proof} (of Lemma \ref{lim_b_lemma}.) We find with some work (use e.g.,\eqref{asfqwfvwfqwg}) that 
\begin{align*}
u_p(b_1^*+,b_1^*,b_2^*)&=\frac{c}{b_1^*r}-\frac{1}{1-b_1^*}\left(\alpha_1(\bar{\lambda})c_1\left(\frac{1-b_1^*}{b_1^*}\right)^{\alpha_1(\bar{\lambda})}+c_2\alpha_2(\bar{\lambda})\left(\frac{1-b_1^*}{b_1^*}\right)^{\alpha_2(\bar{\lambda})}\right)\\
&=\frac{c}{b_1^*r}+\frac{\alpha_2(\bar{\lambda})(b_1^*\bar{\lambda}-c)}{(1-b_1^*)b_1^*r}-\frac{\alpha_1(\bar{\lambda})-\alpha_2(\bar{\lambda})}{1-b_1^*}c_1\left(\frac{1-b_1^*}{b_1^*}\right)^{\alpha_1(\bar{\lambda})}.
\end{align*}
Suppose $b_1^*\leq c/\bar{\lambda}$. Thus, Lemma \ref{c_1 lemma}(a) implies
\begin{align*}
u_p(b_1^*+,b_1^*,b_2^*)&<\frac{c}{b_1^*r}+\frac{\alpha_2(\bar{\lambda})(b_1^*\bar{\lambda}-c)}{(1-b_1^*)b_1^*r}+\frac{(\alpha_1(\bar{\lambda})-\alpha_2(\bar{\lambda}))(b_1^*\bar{\lambda}-c)}{(1-b_1^*)b_1^*r}\\
&=\frac{(1-b_1^*)c-\alpha_1(\bar{\lambda})c}{(1-b_1^*)b_1^*r}+\frac{\alpha_1(\bar{\lambda})\bar{\lambda}}{(1-b_1^*)r}\rightarrow -\infty,
\end{align*}
as $b_1^*\searrow 0$, since $\alpha_1(\bar{\lambda})>1$, and the first result follows.

Now suppose $b_1^* \geq c/\bar{\lambda}$. Then Lemma \ref{c_1 lemma}(b) implies
\begin{align}
u_p(b_1^*+,b_1^*,b_2^*)&>\frac{c}{b_1^*r}+\frac{\alpha_2(\bar{\lambda})(b_1^*\bar{\lambda}-c)}{b_1^*(1-b_1^*)r}-\frac{\alpha_1(\ubar{\lambda})(\bar{\lambda}-\ubar{\lambda})}{(1-b_1^*)r}+\frac{(\alpha_1(\bar{\lambda})-\alpha_2(\bar{\lambda}))(b_1^*\bar{\lambda}-c)}{(1-b_1^*)b_1^*r}\nonumber\\
&=\frac{c}{b_1^*r}+\frac{\left(\alpha_1(\bar{\lambda})\left(\bar{\lambda}-\frac{c}{b_1^*}\right)-\alpha_1(\ubar{\lambda})(\bar{\lambda}-\ubar{\lambda})\right)}{(1-b_1^*)r}.
\label{lower_bound_up}
\end{align}
Fix a sufficiently small $\epsilon>0$. Then, $\alpha_1(\bar{\lambda})>1$ implies that there exists a $\bar{b}\in (0,1)$ such that $\alpha_1(\bar{\lambda})\left(\bar{\lambda}-\frac{c}{b_1^*}\right)>\bar{\lambda}-c+\epsilon$ for any $b_1^*\in(\bar{b},1)$. Thus, \eqref{existence_thm_cond} implies that for $b_1^*\in(\bar{b},1)$, we have that
\begin{align*}
u_p(b_1^*+,b_1^*,b_2^*)&>\frac{c}{b_1^*r}+\frac{\left(\bar{\lambda}-c+\epsilon-\alpha_1(\ubar{\lambda})(\bar{\lambda}-\ubar{\lambda})\right)}{(1-b_1^*)r}\\
&\geq \frac{c}{b_1^*r}+\frac{\epsilon}{(1-b_1^*)r}\to\infty,
\end{align*}
for $b_1^*\nearrow 1$. Hence, the second result follows.
\end{proof}

\section{Results for the proof of Theorem \ref{thm_eq_cand_ex}}\label{app:ODE-exist} 
Throughout this section we consider the setting of the proof of Theorem \ref{thm_eq_cand_ex}. 
Particularly, we here consider a pair $(b_1^*,b_2^*)$ such that \eqref{mainthmcond2} and \eqref{mainthmcond4} hold. 
We also rely on condition \eqref{existence_thm_cond}.

\begin{lemma}\label{v<c/r}
It holds that $v(p)<c/r$ for $p\in (b^*_1,1)$.
\end{lemma}
\begin{proof}
Let us first prove $v(p)\leq c/r$ by contradiction. Assume there exists a $\tilde{p}\in (b^*_1,1)$ such that $v(\tilde{p})>c/r$. Then $v\in \mathcal{C}^{1}(b_1^*,1)$ and $v(1)=c/r$ implies that there exists $\hat{p}\in (\tilde{p},1)$ with $v(\hat{p})>c/r$ such that $v$ attains a local maximum at $\hat{p}$. 
Using the ODE in \eqref{ODE_Employee} we find $v_{pp}(\hat{p}-)\geq v_{pp}(\hat{p}+)>0$, which is a contradiction.

We continue to prove $v(p)< c/r$ by contradiction. For this purpose,
assume that $v(\tilde{p})= c/r$ for some $\tilde{p}\in(b^*_1,1)$. Then by $v\leq c/r$ we have that $v_p(\tilde{p})=0$. We have two cases:
\begin{itemize}
\item If $\tilde{p}\in(b_1^*,b_2^*]$ then the ODE \eqref{ODE_Employee} implies that $v_{pp}(\tilde{p}-)>0$, which contradicts $v\leq c/r$.
\item If $\tilde{p}\in (b_2^*,1)$, using the ODE we find $v(p)=c/r$ for $p\in [b^*_2,\tilde{p}]$ (cf. also \eqref{general-sol-v}). 
Since $v\in \mathcal{C}^1(b_1^*,1)$, we obtain $v_{pp}(b^*_2-)>0$ as above, which again contradicts $v\leq c/r$.
\end{itemize}
\end{proof}

\begin{lemma}\label{v>0}
It holds that 
(a) $v_p(b^*_1+)>0$, and 
(b) $v_p(b^*_2)>0$.
\end{lemma}
\begin{proof}
We will only prove the first statement, since the second statement follows using analogues arguments. Suppose that $v_p(b^*_1+)\leq 0$. We will show that this implies that $v$ has a local minimum below zero, i.e., there exists a point $\tilde{p} \in (b^*_1,1)$ such that $v_p(\tilde{p})=0, v_{pp}(\tilde{p}+)\geq 0$ and $v(\tilde{p})< 0$. 
This contradicts the ODE in \eqref{ODE_Employee} since $c-(\lambda^*(\tilde{p}+)-\ubar{\lambda})^2\geq 0$ (by \eqref{existence_thm_cond}). 
We have three cases:
\begin{itemize}
\item If $v_p(b^*_1+)<0$, then $v(1)=\frac{c}{r}$ and continuity immediately imply that $v$ has a local minimum below zero.
\item If $v_p(b^*_1+)=0,c-(\bar{\lambda}-\ubar{\lambda})^2>0$, then the ODE in \eqref{ODE_Employee} implies that $v_{pp}(b^*_1+)<0$. Analogously to the first case this implies that $v$ has a local minimum below zero. 
\item If $v_p(b^*_1+)=0,c-(\bar{\lambda}-\ubar{\lambda})^2=0$, then we find using the ODE that
$v(p)=0$ for $p\in [b^*_1,b^*_2)$. Using also $v\in \mathcal{C}^1(b^*_1,1)$ and the ODE we conclude that $v(b^*_2)=v_p(b^*_2)=0$ and $v_{pp}(b^*_2+)<0$. With $v(1)=\frac{c}{r}$ this implies that $v$ has a local minimum below zero.
\end{itemize}
\end{proof}

\begin{lemma}\label{inc_v}
It holds that $v_p(p)>0$ for $p\in (b^*_1,1)$.
\end{lemma}

\begin{proof}
We show that $v_p(p)>0$ for $p\in (b^*_1,b^*_2)$ by contradiction. The remaining case can be proved using analogues methods. 
To this end, assume that $p_1\in (b^*_1,b^*_2)$ is the smallest point such that $v_p(p_1)=0$. We consider three cases:
\begin{itemize}

\item If $v_{pp}(p_1)=0$, then the ODE in \eqref{ODE_Employee} implies that $v$ is 
 constant on $(p_1,b^*_2)$, which is a contradiction to $v\in \mathcal{C}^1(b^*_1,1)$ and Lemma \ref{v>0}(b).

\item If $v_{pp}(p_1)>0$, then $p_1$ is a local minimum and Lemma \ref{v>0}(a) implies that $p_1$ cannot be the first point with $v_p=0$.

\item Consider the case $v_{pp}(p_1)<0$. Then since $v_p(b^*_2)>0$ and $v\in\mathcal{C}^2(b^*_1,b^*_2)$, we see that there must exist a second $p_2\in(p_1,b^*_2)$ such that $v_p(p_2)=0$ and $v(p_2)$ is a local minimum. Let $p_2$ be the first such a point. Then it is easy to see, that $v_{pp}(p_2)\geq 0$ and $v(p_2)<v(p_1)$. However, using $v_p(p_1)=v_p(p_2)=0$, $v_{pp}(p_1)<0 \leq v_{pp}(p_2)$ and the ODE we find the contradiction
\begin{align*}
v(p_1)<\frac{c-(\bar{\lambda}-\ubar{\lambda})^2}{r}\leq v(p_2).
\end{align*}

\end{itemize}

\end{proof}

\begin{lemma}\label{dec_lem}
Let $f(p):=p(1-p)v_p(p)$. Then
\begin{enumerate}[(a)]
\item $f'(p_1)<0$ for any $p_1\in (b^*_1,b^*_2)\cup (b^*_2,1)$ satisfying $rv(p_1)\leq c-(\lambda^*(p_1)-\ubar{\lambda})^2$, 
\item $f'(b_2^*-)<0$,
\item $f''(p_2)>0$ for any $p_2\in (b^*_1,b^*_2)$ satisfying $f'(p_2)=0$.
\end{enumerate}
\end{lemma}

\begin{proof}
Let us prove the statement in (a). With the help of the ODE, we observe that

\begin{align}
f'(p)&=v_{pp}p(1-p)+(1-p)v_p-pv_p\notag \\
&=\frac{2(rv-c+(\lambda^*(p)-\ubar{\lambda})^2)}{p(1-p)(\lambda^*)^2(p)}-v_p.\label{dec_lem_help}
\end{align}
Hence, (a) follows from Lemma \ref{inc_v}. Let us prove (b). Using, \eqref{KL:aksldawrv}, \eqref{v_jump}, \eqref{dec_lem_help} and $v\in \mathcal{C}^1(b_1^*,1)$, we find
\begin{align*}
f'(b_2^*-)&=\frac{1}{b_2^*(1-b_2^*)}\left(-\frac{2r(\bar{\lambda}-\ubar{\lambda})}{\alpha_1(\ubar{\lambda})\bar{\lambda}^2\ubar{\lambda}}+\frac{2(\bar{\lambda}-\ubar{\lambda})^2}{\bar{\lambda}^2}-\frac{\bar{\lambda}-\ubar{\lambda}}{\ubar{\lambda}}\right)\\¨
&<\frac{\bar{\lambda}-\ubar{\lambda}}{b_2^*(1-b_2^*)}\left(\frac{2\ubar{\lambda}(\bar{\lambda}-\ubar{\lambda})-\bar{\lambda^2}}{\bar{\lambda}^2\ubar{\lambda}}\right)<0.
\end{align*}
We continue to prove (c). Using \eqref{dec_lem_help}, we find
 \begin{align*}
 f''(p)&=\frac{2rv_p}{p(1-p)\bar{\lambda}^2}-\frac{2(rv-c+(\bar{\lambda}-\ubar{\lambda})^2)(1-2p)}{p^2(1-p)^2\bar{\lambda}^2}-v_{pp}\notag\\
 &=\frac{2rv_p}{p(1-p)\bar{\lambda}^2}-\frac{2(rv-c+(\bar{\lambda}-\ubar{\lambda})^2)(1-2p)}{p^2(1-p)^2\bar{\lambda}^2}+\frac{2v_p}{p}-\frac{2(rv-c+(\bar{\lambda}-\ubar{\lambda})^2)}{p^2(1-p)^2\bar{\lambda}^2}
 \notag\\&=\frac{2rv_p}{p(1-p)\bar{\lambda}^2}-\frac{4(rv-c+(\bar{\lambda}-\ubar{\lambda})^2)}{p^2(1-p)\bar{\lambda}^2}+\frac{2v_p}{p}\\
 &=\frac{2rv_p}{p(1-p)\bar{\lambda}^2}-\frac{2f'(p)}{p}.
\end{align*}
Using $f'(p_2)=0$ and Lemma \ref{inc_v} we thus obtain $f''(p_2)=\frac{2rv_p(p_2)}{p_2(1-p_2)\bar{\lambda}^2}>0$. 
\end{proof}
\begin{proposition}\label{dec_f}
It holds that $f(p)=p(1-p)v_p(p)$ is strictly decreasing for $p\in (b^*_1,1)$. 
\end{proposition}
\begin{proof}
By Lemma \ref{dec_lem}(a), the statement holds for $p\in (b^*_2,1)$, since $v< \frac{c}{r}$ (Lemma \ref{v<c/r}) and 
$\lambda^*-\ubar{\lambda}=0$ on $(b^*_2,1)$. We prove $f'(p)<0$ for $p\in(b_1^*,b^*_2)$ by contradiction. 
Note that $f'(b_1^*+)<0$, by $0=r v(b^*_1)\leq c-(\bar{\lambda}-\ubar{\lambda})^2$ and Lemma \ref{dec_lem}(a). 
For this purpose, let $\tilde{p}\in (b^*_1,b^*_2)$ be a point such that $f'(\tilde{p})=0$, which is then a local minimum (by Lemma \ref{dec_lem}(c)).  
Hence, we obtain $f'(p)\geq 0$ for $p\in(\tilde{p},b_2^*)$ (cf. Lemma \ref{dec_lem}(c)).
This is a contradiction to Lemma \ref{dec_lem}(b).
\end{proof}
\begin{lemma}\label{prop_b_uppar_bound}
It holds that $b^*_1<c/\bar{\lambda}$.
\end{lemma}
\begin{proof}
We prove the statement by contradiction. To this end assume $b^*_1 \geq c/\bar{\lambda}$. Then we can use the calculations in the proof of Lemma \ref{lim_b_lemma} to arrive at \eqref{lower_bound_up}, which with \eqref{mainthmcond2} gives us
\begin{align}
0=u_p(b^*_1)&>\frac{c}{b_1^*r}+\frac{b_1^*\left(\alpha_1(\bar{\lambda})\left(\bar{\lambda}-\frac{c}{b_1^*}\right)-\alpha_1(\ubar{\lambda})(\bar{\lambda}-\ubar{\lambda})\right)}{b_1^*(1-b_1^*)r}\notag\\
&=\frac{(1-\alpha_1(\bar{\lambda}))c+b_1^*\left(\alpha_1(\bar{\lambda})\bar{\lambda}-\alpha_1(\ubar{\lambda})(\bar{\lambda}-\ubar{\lambda})-c\right)}{b_1^*(1-b_1^*)r}.\label{lower_bound_up_help}
\end{align}
Related to the numerator we introduce the function
\begin{align*}
f(b)=(1-\alpha_1(\bar{\lambda}))c+b\left(\alpha_1(\bar{\lambda})\bar{\lambda}-\alpha_1(\ubar{\lambda})(\bar{\lambda}-\ubar{\lambda})-c\right).
\end{align*}
It is easily verified that $f$ is increasing and that $f\left(c/\bar{\lambda}\right)\geq 0$ (use e.g., \eqref{existence_thm_cond}). This is a contradiction to $u_p(b^*_1)=0$ and the statement follows.
\end{proof}

\begin{proposition}\label{u>0_prop}
It holds that $u(p)\geq 0$ for $p\in [0,1]$.
\end{proposition}
\begin{proof}
 We prove that $u(p)> 0$ for $p\in (b^*_1,1]$ by contradiction. By definition, $u(p)=0$ for $p\leq b^*_1$.
Lemma \ref{prop_b_uppar_bound} establishes $b^*_1< c/\bar{\lambda}$. Thus the ODE \eqref{u_ODE} and $u_p(b^*_1+)=0$ imply that $u(p)$ is strictly increasing and convex on
$$
\{p\in(b^*_1,1):ru(\hat{p})+c-\hat{p}\lambda^*(\hat{p})>0, \forall \hat{p}\in (b^*_1,p) \},
$$ 
which is non-empty. Suppose, in order to obtain a contradiction, that $\tilde{p}\in (b^*_1,1]$ is the smallest point such that $u(\tilde{p})=0$. It can then be verified that $\tilde{p} > c/\bar{\lambda}$. 

Let us consider the case $\tilde{p}<b^*_2$. Recall that \eqref{lower_bound_up} holds for any $c/\bar{\lambda} \leq b_1^*<b_2^*<1$. 
Hence, using $\tilde p$ instead of $b^*_1$ in \eqref{lower_bound_up} and the same reasoning that lead to \eqref{lower_bound_up_help} gives 
\begin{align*}
u_p(\tilde p+; \tilde p, b_2^*) > \frac{(1-\alpha_1(\bar{\lambda}))c+\tilde{p}\left(\alpha_1(\bar{\lambda})\bar{\lambda}-\alpha_1(\ubar{\lambda})(\bar{\lambda}-\ubar{\lambda})-c\right)}{\tilde{p}(1-\tilde{p})r}.
\end{align*}
Note that that we have $u_p(\tilde p+; \tilde p, b_2^*)=u_p(\tilde p; b_1^*, b_2^*):=u_p(\tilde p)$, since both functions 
$u(\cdot; \tilde p, b_2^*)$
and 
$u(\cdot; b_1^*, b_2^*)$ satisfy the ODE in \eqref{u_ODE} on $(\tilde p,1)$ with the same boundary conditions, in particular 
$u(\tilde p; \tilde p, b_2^*)=0$ (by \eqref{u_ODE}) and
$u(\tilde p; b_1^*, b_2^*)=0$ (by the contradiction assumption).  
Hence, using $\tilde{p} > c/\bar{\lambda}$ and arguments analogous to those after \eqref{lower_bound_up_help} we find that $u_p(\tilde{p})> 0$.  
However, $u_p(\tilde{p})> 0$ is a contradiction to the definition of $\tilde{p}$ being the smallest point in $(b^*_1,1]$ where $u(\tilde{p})=0$.

Let us consider the case $\tilde{p}\geq b^*_2$. Then the contradiction $u_p(\tilde{p})> 0$ is obtained in a similar way. More precisely, using the ODE (cf. \eqref{value_employer}) with the boundary conditions $u(\tilde{p})=0$ and $u(1)=\frac{\ubar{\lambda}-c}{r}$, we obtain
\begin{align*}
u(p)=c_3p\left(\frac{1-p}{p}\right)^{\alpha_1(\ubar{\lambda})}+\frac{p\ubar{\lambda}-c}{r},\quad \text{for }p>\tilde{p},
\end{align*}
where 
\begin{align*}
c_3=-\frac{\tilde{p}\ubar{\lambda}-c}{\tilde{p}r}\left(\frac{1-\tilde{p}}{\tilde{p}}\right)^{-\alpha_1(\ubar{\lambda})}.
\end{align*}
Using also $\tilde{p} > c/\bar{\lambda}$, \eqref{existence_thm_cond}, and arguments analogous to those after \eqref{lower_bound_up_help}, 
we find with some work that 
\begin{align*}
u_p(\tilde{p})=\frac{c}{\tilde{p}r}+\frac{\alpha_1(\ubar{\lambda})(\tilde{p}\ubar{\lambda}-c)}{\tilde{p}(1-\tilde{p})r}>0.
\end{align*}

\end{proof}

\bibliographystyle{abbrv}
\bibliography{Bibl}

\end{document}